\documentclass[final,reqno]{siamart1116}
\usepackage[utf8]{inputenc}
\usepackage{amsfonts}
\usepackage{booktabs}
\usepackage{color}
\usepackage{graphicx}

\title{A perfectly matched layer approach for radiative transfer in highly scattering regimes
\thanks{Submitted to the editors \today.}}
\author{Herbert Egger%
\thanks{Numerical Analysis and Scientific Computing, Department of Mathematics, TU Darmstadt, Dolivostr. 15, 64293 Darmstadt, Germany. 
\email{egger@mathematik.tu-darmstadt.de}}
\and Matthias Schlottbom%
\thanks{Department of Applied Mathematics, University of Twente,
P.O. Box 217, 7500 AE Enschede, The Netherlands.
\email{m.schlottbom@utwente.nl}}
}
\newcommand{\TheTitle}{A perfectly matched layer approach}
\newcommand{\TheAuthors}{H. Egger, M. Schlottbom}
\headers{\TheTitle}{\TheAuthors}
\def\d{\,d}
\def\s{s}
\def\r{r}
\def\k{k}
\def\n{n}
\def\sn{s \cdot \n}
\def\sgrad{s \cdot \nabla}
\def\u{u}
\def\v{v}
\def\w{w}
\def\z{z}
\def\K{K}
\def\C{\mathcal{C}}
\def\Ci{\C^{-1}}
\def\H{H}
\def\q{q}
\def\wu{\widetilde\u}

\def\wz{\widetilde\z}
\def\ula{u^{\ell,\a}}
\def\wla{w^{\ell,\a}}
\def\R{\mathcal{R}}

\def\dR{\partial\mathcal{R}}
\def\S{\mathcal{S}}
\def\D{\mathcal{D}}
\def\P{\mathcal{P}}
\def\T{\mathcal{T}}
\def\W{\mathcal{W}}

\def\SS{\mathbb{S}}
\def\XX{\mathbb{X}}
\def\a{a}
\def\RR{\mathbb{R}}

\def\WW{\mathbb{W}}
\def\VV{\mathbb{V}}
\def\ttM{\mathtt{M}} 
\def\ttS{\mathtt{S}} 
\def\ttR{\mathtt{R}} 
\def\ttB{\mathtt{B}}
\def\ttC{\mathtt{C}}

\def\ttup{\mathtt{w}^+} 
\def\ttum{\mathtt{w}^-} 
\def\ttqp{\mathtt{q}^+} 
\def\ttqm{\mathtt{q}^-} 
\newtheorem{problem}[theorem]{Problem}

\newtheorem{remark}[theorem]{Remark}

\newcommand{\tnorm}[1]{{\left\vert\kern-0.25ex\left\vert\kern-0.25ex\left\vert #1 
    \right\vert\kern-0.25ex\right\vert\kern-0.25ex\right\vert}}
\newenvironment{myass}
               {\list{}{\listparindent 0.5ex%
                        \itemindent    \listparindent
                        \topsep        1.5ex}%
               }
               {\endlist}
\begin{document}
\maketitle
\begin{abstract}
We consider the numerical approximation of boundary conditions in radiative transfer problems by a perfectly matched layer approach.
The main idea is to extend the computational domain by an absorbing layer and to use an appropriate reflection 
boundary condition at the boundary of the extended domain. A careful analysis shows that the consistency error 
introduced by this approach can be made arbitrarily small by increasing the size of the extension domain or the 
magnitude of the artificial absorption in the surrounding layer.
A particular choice of the reflection boundary condition allows us to circumvent the half-space 
integrals that arise in the variational treatment of the original vacuum boundary conditions and which destroy 
the sparse coupling observed in numerical approximation schemes based on truncated spherical harmonics expansions. 
A combination of the perfectly matched layer approach with a mixed variational formulation and a PN-finite element approximation 
leads to discretization schemes with optimal sparsity pattern and provable quasi-optimal convergence properties. 
As demonstrated in numerical tests these methods are accurate and very efficient for radiative transfer in the scattering regime. 
\end{abstract}
\section{Introduction}
Radiative transfer problems arise in a variety of applications, such as astrophysics, meteorology, 
nuclear reactor physics, or medical treatment and imaging; we refer to \cite{Arridge99,CaseZweifel67,Chandrasekhar60,DuderstadtMartin79,Modest,TervoKolmonen2002} for examples and further references.
In this paper, we consider a particular aspect of such models, namely the efficient numerical treatment of boundary conditions. 
For ease of presentation, we consider a mono-chromatic and stationary model problem
\begin{alignat}{2}
\sgrad \u(\r,\s) + \mu(\r) \u(\r,\s) 
   &= \int_\S k(\r,\s\cdot \s')\u(\r,\s')d\s' + \q(\r,\s)  
   \quad \text{in } \R \times \S, \label{eq:rte1} \\
\intertext{together with vacuum (homogeneous inflow) boundary conditions}
             \u(\r,\s) &= 0          
   \qquad \qquad \quad \text{on } \dR \times \S \text{ with } \s \cdot \n(\r)<0. \label{eq:rte2}
\end{alignat}
Here, $\n(\r)$ is the outer unit normal vector on $\partial\R$. 
This model describes the transport, absorption, and scattering of particles 
propagating through a bounded domain $\R$ which is filled by some background medium and surrounded by vacuum. 
$\u=\u(\r,\s)$ denotes the density of particles at position $\r \in \R$ 
traveling in direction $\s \in \S$, the coefficients $\mu(\r)$ and $\k(\r,\s \cdot \s')$ describe 
the attenuation and scattering properties of the medium and $q(\r,\s)$ is a given source density. 
Due to the inherent tensor product structure of the phase space $\R \times \S$, it seems natural to expand
the density $\u(\r,\s)$ into a series 
\begin{align} \label{eq:fourier}
\u(\r,\s) = \sum\nolimits_n \u_n(\r) \H_n(\s),
\end{align}
which allows to formally recast the radiative transfer equation \eqref{eq:rte1} as an infinite system of coupled 
partial differential equations for the moments $u_n(\r)$. 
A particularly well-suited choice for the basis functions $\H_n$ are the spherical harmonics,
since they form a complete orthogonal system in $L^2(\S)$ corresponding to the eigenfunctions of the scattering operator.
Moreover, the product $\s \H_n(\s)$ can be expressed as a finite linear combination of spherical harmonics $\H_m$, which 
leads to a sparse coupling of the moment equations arising from the spherical harmonics expansion; 
let us refer to \cite{Arridge99,MarchukLebedev86} for details.
In the highly scattering regime, the density $\u(\r,\s)$ is a smooth function of $\s$, which results 
in a fast decay of the moments $\u_n(\r)$ in the spherical harmonics expansion \eqref{eq:fourier} with $n \to \infty$. 
A good approximation for the density can thus already be obtained by a truncated series $\u_N(\r,\s) = \sum_{n=0}^N \u_n(\r) \H_n(\s)$ with $N$ small. Inserting this ansatz into equation \eqref{eq:rte1} leads to the well-known $P_N$-approximations, 
which have been used successfully for theoretical investigations and for the design of numerical approximation schemes; 
we refer to \cite{CaseZweifel67,DuderstadtMartin79,Vladimirov61} and \cite{LewisMiller84,MarchukLebedev86,WrightArridgeSchweiger09} for details. 
While the formal derivation of the $P_N$-approximation for the radiative transfer equation \eqref{eq:rte1} is rather straight forward, 
the correct approximation of the vacuum boundary conditions \eqref{eq:rte2} has been subject of controversial discussion for many years; 
see \cite{Modest} for a comprehensive overview. 
A systematic treatment is possible by variational formulations \cite{Agoshkov98,EggerSchlottbom12,ManResSta00}, 
in which the boundary conditions \eqref{eq:rte2} give rise to half-space integrals of the form 
\begin{align} \label{eq:halfspace}
\int_{\dR} \int_{\S : \s \cdot \n(\r)<0} \u(\r,\s) \v(\r,\s) |\s \cdot \n|  \d\s \d\r;
\end{align}
here, $\v$ denotes the test function in the variational formulation. 
The appropriate boundary conditions for the $P_N$-approximation can then be obtained rigorously by
Galerkin projection of the underlying variational principle. 
Let us note that the half-space integrals \eqref{eq:halfspace} no longer have a tensor product structure,
which actually leads to a dense coupling of almost all moments $\u_n(\r)$ in the spherical harmonics expansion of the system \eqref{eq:rte1}--\eqref{eq:rte2}.
Numerical methods based on $P_N$-approximations, therefore, suffer from a dense coupling of the moment equations 
originating from the non-tensor product structure of the boundary conditions.
This not only complicates the implementation but also negatively affects the performance of corresponding discretization methods.
In this paper, we propose a strategy to overcome these problems associated with the numerical approximation of the boundary conditions \eqref{eq:rte2}. In the spirit of the \emph{perfectly matched layer} approach, which has been successfully used in the context of acoustic and electromagnetic wave propagation  \cite{Berenger94,Hagstrom99}, we proceed as follows:
\begin{itemize}
\item[(i)] In a first step, the domain $\R$ is extended by an absorbing but non-scattering layer of thickness $\ell>0$ with absorption coefficient $\a>0$. If vacuum boundary conditions are used at the outer boundary, this yields an equivalent formulation of problem \eqref{eq:rte1}--\eqref{eq:rte2} on an extended domain $\R^\ell$, whose solution $\ula$ coincides with $\u$ when restricted to $\R$. Due to the presence of the absorbing layer $\R^\ell \setminus \R$, the solution $\ula$ decays exponentially towards $\partial\R^\ell$.
\item[(ii)] In a second step, the vacuum boundary condition at the boundary of the extended domain 
is replaced by a reflection boundary condition. Since $\ula$ is already small at $\partial\R^\ell$,
this introduces a minor perturbation that can be controlled by the absorption parameter $a$ and the thickness $\ell$ of the absorbing layer in terms of an estimate of the form $\|\ula - \wla\| = O(e^{-\ell a})$, where $\wla$ is the solution of the problem with reflection boundary condition.
\end{itemize}
An appropriate choice of the thickness $\ell$ and the absorption coefficient $a$ in the surrounding layer $\R^\ell \setminus \R$ thus allows to obtain solutions $\wla$ of a perturbed problem, whose restriction to $\R$ 
approximates the original solution $\u$ with any desired accuracy. 
The rigorous analysis of this approach will be the main topic of the first part of the paper.
In the second part of the manuscript, we consider the numerical approximation of the problem with reflection boundary conditions outlined in step (ii). Based on the ideas of \cite{EggerSchlottbom12}, we investigate in detail the Galerkin approximation of a mixed variational formulation of the perturbed problem. 
The possible extension of our analysis to other approaches is briefly discussed at the end of the manuscript. 
The main new contributions are the following:
\begin{itemize}
 \item[(iii)] A specific choice of the reflection boundary condition allows us to extend the variational formulation given in \cite{EggerSchlottbom12} to the perturbed problem discussed in step (ii) and to avoid the half-space integrals \eqref{eq:halfspace}. 
 A careful analysis of the variational problem allows us to establish its well-posedness.
 \item[(iv)] Under a mild compatibility condition of the approximation spaces, the Galerkin approximation of the mixed variational method
 leads to discretization schemes with provable convergence properties. A full analysis of the general approach is given
 and as a particular example, we discuss in some detail the extension of the $P_N$-finite element approximation considered in \cite{EggerSchlottbom12,LewisMiller84,WrightArridgeSchweiger09}. Due to the absence of the half-space integrals \eqref{eq:halfspace}, which are eliminated by the particular reflection boundary conditions, the resulting linear systems can be shown to have an optimal sparsity and a tensor product structure that allows for a very efficient solution.
\end{itemize}
For illustration of theoretical results and in order to demonstrate the efficiency of our approach, we report about some 
numerical tests for the proposed $P_N$-finite element approximation with the perfectly matched layer approach at the end of the manuscript.
Before we proceed, let us mention a recent paper \cite{PowellCoxArridge18}, where the authors considered a somewhat related idea. In this work, a PML approach is used to obtain a problem with periodic boundary conditions in space which in turn can be discretized efficiently by Fourier series.  
The efficiency of the resulting pseudospectral approximation was illustrated by numerical tests. A full analysis of 
this approach is not available yet, but might be possible with the arguments presented here. 

\medskip 

The remainder of the manuscript is organized as follows: 
In Section~\ref{sec:prelim}, we introduce our notation and main assumptions and we recall some preliminary 
results about well-posedness of the radiative transfer equation. 
In Section~\ref{sec:extension}, we then formulate and analyze the problem in step (i) that arises from 
extension of the computational domain by an absorbing layer.
Section~\ref{sec:reflection} deals with the analysis of the perturbed problem with reflection boundary conditions described in step (ii). 
In Section~\ref{sec:variational} we consider step (iii) of our approach by deriving a mixed variational 
formulation of the problem with reflection boundary conditions. In addition, we investigate its systematic Galerkin approximation
and establish rigorous error estimates.
In Section~\ref{sec:pnfem}, we address point (iv) by considering an extension of  
the mixed $P_N$-finite element method proposed in \cite{EggerSchlottbom12} to the setting considered here. 
We state a basic compatibility condition of the approximation spaces and discuss some further properties of the method.
In Section~\ref{sec:num}, we comment on the efficient implementation of this discretization scheme 
and then present some numerical tests for illustration of the efficiency of the method.
We close with a short discussion and indicate some possible extensions.
\section{Preliminaries and notation}\label{sec:prelim}
Let us start by introducing our notation and basic assumptions that will allow us to guarantee the well-posedness of the radiative transfer problem under consideration. 
Throughout the manuscript, we assume that 
\begin{myass}
 \item[(A1)] the domain $\R \subset \RR^3$ is bounded and convex and we let 
 $\S=\S^2$ be the unit sphere in $\RR^3$. The phase space is denoted by $\D=\R \times \S$. 
\end{myass}
Note that the assumption about convexity of $\R$ is not very restrictive, since one may always extend the domain to a larger ball
if required. 
In our numerical tests, we will actually consider the case $\R \subset \RR^2$ and $\S=\S^2$, which amounts to 
In our numerical tests, we will consider the case $\R \subset \RR^2$ and $\S=\S^2$, which amounts to 
a problem with invariance in one spatial direction.
As usual, we decompose the boundary $\partial\D=\partial\R\times\S$ via 
\begin{align*}
\partial\D_\pm=\{ (\r,\s) \in \partial\D : \pm \s \cdot \n(\r)>0\}
\end{align*}
into an inflow part $\partial\D_-$ and an outflow part $\partial\D_+$. 
Let us recall at this point that boundary conditions \eqref{eq:rte2} are required only for the inflow part $\partial\D_-$ of the boundary.
\subsection{Function spaces}
For any sufficiently regular submanifold $M \subset \RR^n$ and any $1 \le p \le \infty$, 
we denote by $L^p(M)$ the usual Lebesgue space of functions on $M$ and we use $(u,v)_M = \int_M u v \, dM$ 
to denote the scalar product of $L^2(M)$.
Following the notation of \cite{DautrayLions6}, we further write
\begin{align*}
W^{p}(\D) = \{ \u \in L^p(\D) : \sgrad \u \in L^p(\D)\}
\end{align*}
for the Sobolev space of functions with integrable weak directional derivatives and finite norm given by
\begin{align*}
\|u\|_{W^p(\D)}^p = \|u\|_{L^p(\D)}^p + \|\sgrad u\|_{L^p(\D)}^p.
\end{align*}
Let us recall that functions $\u  \in W^p(\D)$ possess well-defined traces on $\partial\D$ in some weighted $L^p$ spaces; see e.g. \cite{Agoshkov98,DautrayLions6}. 
For a.e. $(\r,\s) \in \partial\D$, we may thus define 
\begin{align*}
\u_{\pm}(\r,\s) = 
\begin{cases} 
  \u(\r,\s), & \pm \s \cdot \n(\r) > 0, \\
  0              , & \text{else}.
\end{cases} 
\end{align*}
This induces a natural splitting $u = u_- + u_+$ of the boundary values on $\partial\D$ 
into an ingoing trace $\u_-$ and an outgoing trace $\u_+$,
which are, respectively, supported on the corresponding parts $\partial\D_-$ and $\partial\D_+$ of the boundary.
By the divergence theorem and a density argument, one can see that
\begin{align} \label{eq:ipp}
(\sgrad \u,\v)_\D = -(\u, \sgrad \v)_\D + (\sn \u, \v)_{\partial\D},
\end{align}
holds for all functions $\u,\v$ with sufficient smoothness and integrability properties. 
This integration-by-parts formula motivates the definition of weighted trace spaces
\begin{align*}
L^p(\partial\D;|\sn|) = \{g : \partial\D \to\RR \quad \text{with }  \int_{\partial\D} |g(\r,\s)|^p |\sn| \d(\r,\s) < \infty\},
\end{align*}
which are strictly smaller than the natural trace spaces of $W^p(\D)$; see \cite{Agoshkov98,DautrayLions6} for details.
For any $u \in W^p(\D)$ with regular ingoing trace $u_- \in L^p(\partial\D;|\sn|)$, one can deduce from \eqref{eq:ipp} with $v=|u|^{p-2} u$ and some elementary manipulations that
\begin{align} \label{eq:outoing-bound}
\|u_+\|_{L^p(\partial\D;|\sn|)}^p 
&= \|u_-\|_{L^p(\partial\D;|\sn|)}^p + p (\sgrad \u, |u|^{p-2} u)_{\D} \\
&\le \|u_-\|_{L^p(\partial\D;|\sn|)}^p + \|\sgrad \u\|_{L^p(\D)}^p + \tfrac{p-1}{p} \|u\|_{L^p(\D)}^p. \notag
\end{align}
Hence, the norm of the outgoing trace $u_+$ can be controlled in terms of the norm of the ingoing trace $u_-$ and the norm of the solution $u$ in $W^p(\D)$.
\subsection{Basic assumptions and well-posedness}
In order to ensure the well-posedness of the radiative transfer problem \eqref{eq:rte1}--\eqref{eq:rte2}, we will make 
the following structural assumptions on the model parameters, namely
\begin{myass}
\item[(A2)] $\mu\in L^\infty(\R)$ with $0 \le \mu(\r) \le \overline{\mu}$;
\item[(A3)] $k : L^\infty(\R \times (-1,1))$ with $0 \le \k(\r,\theta) \le \overline\k$ and $\int_\S \k(\r;\s \cdot \s') \d\s'\leq\mu(\r)$. 
\end{myass}
These rather general conditions are motivated by physical considerations.
The well-posedness of problem~\eqref{eq:rte1}--\eqref{eq:rte2} is a special case of the following result, which also covers inhomogeneous boundary conditions. 
\begin{theorem} \label{thm:1a}
Let (A1)--(A3) hold and let 
$
(\K\u)(\r,\s):= \int_{\S} \k(\r;\s \cdot \s') \u(\r,\s') \d\s'
$ 
denote the scattering operator with kernel function $k$. 
Then, for any $q \in L^p(\D)$ and any $g \in L^p(\partial\D;|\sn|)$, the radiative transfer problem
\begin{alignat}{4}
\sgrad \u + \mu \u &= \K\u + q &\quad& \text{in } \D,           \label{eq:rte1a}\\
\u_- &= g_-                    &\quad& \text{on } \partial\D,   \label{eq:rte2a}
\end{alignat}
has a unique solution $\u \in W^p(\D)$ and there holds
\begin{align*}
\|\u\|_{W^p(\D)} + \|u_+\|_{L^p(\partial\D;|\sn|)}  
\le C (\|q\|_{L^p(\D)} + \|g_-\|_{L^p(\partial\D;|\sn|)})
\end{align*}
with constant \emph{$C$ depending only on $\overline \mu$ and ${\rm diam}(\R)$}. 
\end{theorem}
\begin{proof}
Existence of a unique solution and the bound for $u$ in the norm of $W^p(\D)$ follows from  \cite[Theorem~1.1 and Theorem~8.3]{EggerSchlottbom2014Lp}. The remaining estimate for the outgoing trace $u_+$ can then be deduced from \eqref{eq:outoing-bound}.
\end{proof}
Note that problem \eqref{eq:rte1}--\eqref{eq:rte2} is just a special case of \eqref{eq:rte1a}--\eqref{eq:rte2a} with $g_-=0$.
Under assumptions (A1)--(A3), the model problem \eqref{eq:rte1}--\eqref{eq:rte2} is therefore well-posed.
\bigskip 
\begin{center}
\sc Part 1: The perfectly matched layer approach 
\end{center}
\medskip 
In the following two sections, we investigate the approximation of \eqref{eq:rte1}--\eqref{eq:rte2} 
by radiative transfer problems on larger domains. We start with an equivalent problem and then introduce a perturbation by 
incorporating a reflection boundary condition. 
\section{Equivalent problems on larger domains} \label{sec:extension}
Problem \eqref{eq:rte1}--\eqref{eq:rte2} describes the propagation of particles through a domain $\R$ surrounded by vacuum. 
We will now show that the domain $\R$ can also be embedded in an absorbing medium without changing the solution. 
For any $\r\in \RR^d \setminus \R$ and $\s \in \S$, we denote by
\begin{align} \label{eq:ell}
\ell(\r,\s) = \inf \{l >0: \r - l \s \in \R\}
\end{align}
the distance of the point $\r$ to the boundary $\partial\R$ of the computational domain along the path with direction $-\s$ starting at $\r$; see Figure~\ref{fig:sketch}. 
Using standard convention, we set $\ell(\r,\s)=\infty$, if the corresponding path does not intersect the boundary $\partial\R$. 
\begin{figure}[ht!]
\begin{center}
    \setlength{\unitlength}{275bp}%
    \begin{picture}(1,0.87544529)%
      \put(0,0){\includegraphics[width=1\unitlength,page=1]{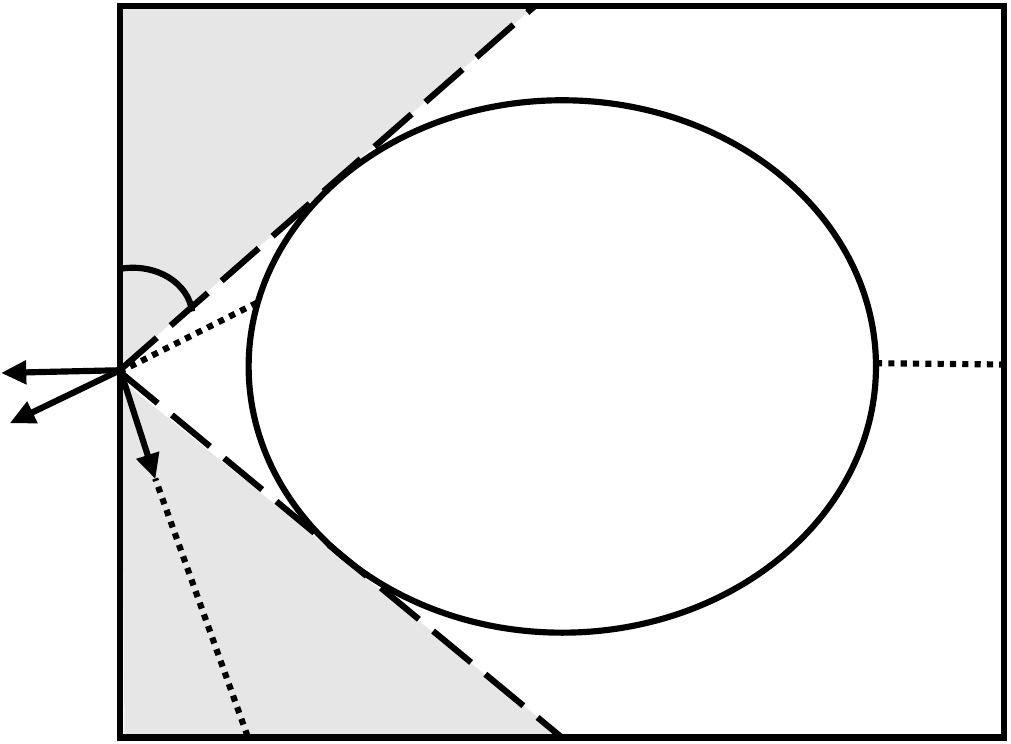}}%
      \put(0.55,0.37){$\R$}%
      \put(0.92533403,0.67){$\R^\ell$}%
      \put(0.92215333,0.38){$\ell$}%
      \put(0.27,0.42){$r-\ell(\r,\s)s$}
      \put(0.01092941,0.39384543){$n(\r)$}%
      \put(0.059167,0.3111482){$\s$}%
      \put(0.09,0.39){$\r$}%
	  \put(0.13,0.42){$\alpha'$}%
 	  \put(0.16,0.27){$\bar \s$}%
 	  \put(0.25,0.02){$\bar \r^*=r+ t^* \bar s$}%
    \end{picture}%
\end{center}
 \caption{
   Sketch of the geometric setup for two spatial dimensions. $\R$ corresponds to the disc. The extended domain $\R^\ell$ corresponds to the bounding rectangle. The distance between $\partial \R$ and $\partial\R^\ell$ is given by $\ell$, and the point $(\r,\s)$ is an element of the outflow boundary $\partial\D^\ell_+$ of the layer such that $\ell(r,s)<\infty$, i.e., $r-\ell(r,s)s\in\partial \R$. 
 Moreover, we have that $\s\cdot\n(\r)\geq\sin(\alpha')$.
   On the other hand, lines through $r$ that pass through the gray area do not intersect $\R$. 
 For instance, the path $t\mapsto r + t \bar s$, $t\in\RR$, does not intersect $\R$. Moreover, there exists a unique $t^*>0$ such that $r^*=r+t^* \bar s\in\partial\R^\ell$, and $(r^*,\bar s)\in\partial\D_+^\ell$, which will become important for the construction of the reflection boundary conditions described in step (ii).
   \label{fig:sketch}}
\end{figure}
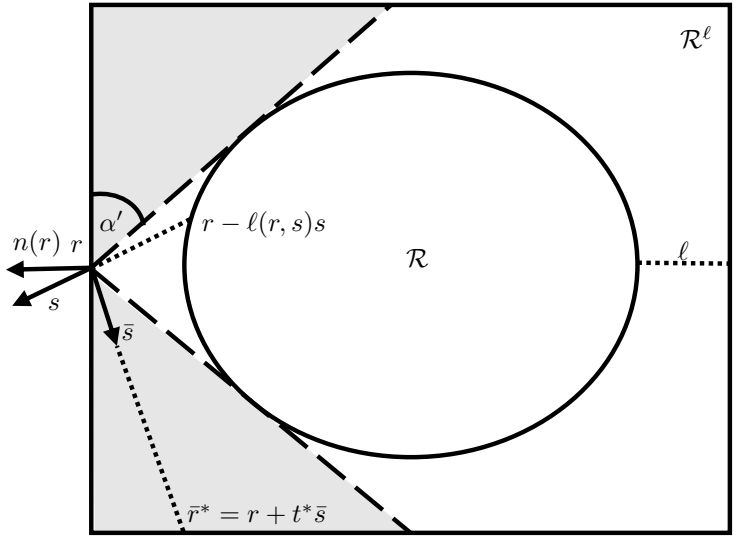
We then consider extensions $\R^\ell$ of the domain $\R$ with the following properties:
\begin{myass}
\item[(A4)] For given $\ell,\eta>0$, let $\R^\ell \subset \RR^3$ be a bounded convex domain with $\overline \R \subset \R^\ell$
 compactly embedded and such that $\ell(\r,\s) \ge \ell$ for a.e. $(\r,\s) \in \partial\D^\ell=\partial\R^\ell\times\S$ and $\ell(\r,\s)=\infty$ for a.e. $(\r,\s) \in \partial\D^\ell$ with $\sn(\r) \leq \eta=:\sin \alpha$.
\end{myass}
\begin{remark} \label{rem:geometry} \rm
Note that $\ell \le {\rm dist}\{\partial\R^\ell,\R\}$ is a lower bound on the thickness of the extension layer $\widetilde \R = \R^\ell \setminus \R$. Moreover, $0 <\eta = \sin \alpha$ yields a lower bound on the angles $\alpha'$ at which beams originating from points $r \in \R$ can hit the boundary $\partial\R^\ell$ and lines in direction $\s$ going through points $r \in \R^\ell \setminus \R$ with $\ell(\r,\s)=\ell(\r,-\s)=\infty$ do not intersect the domain $\R$; see~Figure~\ref{fig:sketch} for illustration. These geometric properties will become important for our analysis below. 
\end{remark}
As a next step, we extend the definition of the model parameters to $\R^\ell$ by
\begin{alignat*}{7}
\mu^{\ell,\a}(\r) &= \mu(\r), 
 &\quad k^\ell(\r,\cdot) &= k(\r,\cdot), 
 &\quad q^\ell(\r,\cdot) &= q(\r,\cdot), &&\qquad \r \in \R,  \\   
\mu^{\ell,\a}(\r) &= \a,
 &\quad k^\ell(\r,\cdot) &= 0,
 &\quad q^\ell(\r,\cdot) &= 0,           &&\qquad \r \in \R^\ell \setminus \R,
\end{alignat*}
and we denote by $\K^\ell$ the scattering operator associated to the kernel $k^\ell$. 
The choice $\a = 0$ means that $\R$ is surrounded by vacuum, while $\a>0$ 
models the case that the original domain is embedded in an absorbing but non-scattering medium. 
On the extended domain $\D^\ell = \R^\ell \times \S$, we then consider the problem 
\begin{alignat}{4}
\sgrad \u^{\ell,\a} + \mu^{\ell,\a} \u^{\ell,\a} &= \K^\ell \u^{\ell,\a} + \q^\ell &\quad& \text{on } \D^\ell, \label{eq:ext1}\\
\u^{\ell,\a}_- &= 0 && \text{on } \partial \D^\ell. \label{eq:ext2}
\end{alignat}
With the same arguments as used for the proof of Theorem~\ref{thm:1a}, one can again obtain the existence of a unique solution. 
Due to the particular definition of the parameters in the extension layer, we obtain some further properties of the solution.
\begin{theorem} \label{thm:1}
Let (A1)--(A4) hold and $\a \ge 0$. Then for any $q \in L^p(\D)$, 
the extended problem \eqref{eq:ext1}--\eqref{eq:ext2} has a unique solution $\u^{\ell,\a}\in W^p(\D^\ell)$,
which can be represented as $\ula = E^{\ell,\a} \u$, where $\u$ is the solution of \eqref{eq:rte1}--\eqref{eq:rte2} and 
where the extension operator $E^{\ell,\a}$ is defined by 
\begin{align*}
(E^{\ell,\a} u)(\r,\s) = 
\begin{cases} 
u(\r,\s), & (\r,\s) \in \D, \\ 
e^{-\a\ell(\r,\s)} u(\r-\ell(\r,\s) \s,\s), & (\r,\s) \in \D^\ell \setminus \D, \ 0 < \ell(\r,\s)<\infty, \\ 
0, & \text{else}.
\end{cases}
\end{align*}
Moreover, $\u^{\ell,\a}|_\D = \u$ and $u^{\ell,\a}(\r,\s)=0$ for $(\r,\s) \in \partial\D^\ell$ with $\ell(\r,\s)=\infty$,
and
\begin{align} \label{eq:decay}
\|\u^{\ell,\a}\|_{L^p(\partial\D^\ell)} \le C e^{-\a \ell}\|q\|_{L^p(\D)}
\end{align}
with constant $C$ depending only on $\overline{\mu}$, ${\rm diam}(\R)$, and the constant $\eta$ in (A4).
\end{theorem}
\begin{proof}
Existence of a unique solution $\u^{\ell,\a}$ follows from \cite[Theorem~1.1 and Theorem~8.3]{EggerSchlottbom2014Lp}.
The remaining assertions are proven by the following lemma.
\end{proof}
\begin{lemma} \label{lem:extension}
Let (A1)--(A4) hold and let $\u \in W^p(\D)$ with $u_-=0$ on $\partial\D$.
Then for any $a \ge 0$, we have $E^{\ell,\a} \u \in W^p(\D^\ell)$ and 
\begin{align*}
\|\sgrad E^{\ell,\a}\u\|_{L^p(\D^\ell \setminus \D)} = \a \|E^{\a,\ell}\u\|_{L^p(\D^\ell \setminus \D)} \le
\a^{1-\frac{1}{p}} \|\u\|_{W^p(\D)}.
\end{align*}
Moreover, $(E^{\ell,\a} \u)(\r,\s)=0$ for any $(\r,\s) \in \partial\D^\ell$ with $\ell(\r,\s)=\infty$ and, therefore,
\begin{align*}
\eta^{1/p} \|E^{\ell,\a} \u\|_{L^p(\partial\D^\ell)} 
&\le \|E^{\ell,a} \u\|_{L^p(\partial\D^\ell;|\sn|)} \\
&\le e^{-\a\ell} \|u_+\|_{L^p(\partial\D;|\sn|)}
\le e^{-\a \ell} \|\u\|_{W^p(\D)}. 
\end{align*}
\end{lemma} 
\begin{proof}
By construction, $\widetilde \u = (E^{\ell,\a} \u)|_{\widetilde\D}$, with $\widetilde \D = \D^\ell \setminus \D$, is a solution to
\begin{alignat}{5}
\sgrad \wu + \a \wu &= 0 && \quad \text{in } \widetilde \D, \label{eq:tilde1}\\
\wu_- &= u_+ && \quad \text{on } \partial\widetilde\D \cap \partial\D 
\quad \text{and} \quad
\wu_- &= 0 && \quad \text{on } \partial\widetilde\D \cap \partial\D^\ell. \label{eq:tilde2}
\end{alignat}
Note that the normal vector pointing out of the layer $\widetilde \R = \R^\ell \setminus \R$ has to be used in the definition of $\wu_\pm$,
while that pointing out of $\R$ is used in the definition of $\u_\pm$. 
From \cite[Theorem~1.2]{EggerSchlottbom2014Lp} with $\nu=1$, $\sigma=\a$, and $f=0$ and noting that $\sgrad \wu=-\a \wu$,
we infer that $\wu \in W^p(\widetilde \D)$ and
\begin{align*}
 \|\sgrad \wu\|_{L^p(\widetilde\D)} = \a \|\wu\|_{L^p(\widetilde\D)} \le \a^{1-\frac{1}{p}} \|\u_+\|_{L^p(\partial\D;|\sn|)},
\end{align*}
From \eqref{eq:outoing-bound} and $\u_-=0$ on $\partial\D$, we deduce that 
\begin{align}\label{eq:trace_u}
	\|\u_+\|_{L^p(\partial\D;|\sn|)} \le \|\u\|_{W^p(\D)},
\end{align}
which proves the first estimate.
Moreover, we have $\wu_+=0=\u_-$ and, by \eqref{eq:tilde2},  $\wu_-=\u_+$ on $\partial\D$.
Hence $\wu = u$ on $\partial\D$, which shows that $E^{\ell,\a} \u$ is continuous across $\partial\D$ in the sense of traces.
Together with $(E^{\ell,\a} \u) |_\D = \u \in W^p(\D)$ and $(E^{\ell,\a} \u)|_{\widetilde\D} = \wu \in W^p(\widetilde\D)$, 
this implies that $E^{\ell,\a} \u \in W^p(\D^\ell)$; see \cite[Remark~2.5]{Agoshkov98}. 
By the definition of the extension and condition (A4), one can see that $E^{\ell,\a}\u(\r,\s)=0$ 
for all $(\r,\s) \in \partial\D^\ell$ with $\ell(\r,\s)=\infty$; cf. Figure~\ref{fig:sketch}.
In addition, one can infer that $|E^{\ell,\a}\u(\r,\s)| \le e^{-\a \ell} |E^{\ell,0}\u(\r,\s)|$ on $\partial\widetilde\D$. 
But since $\sgrad E^{\ell,0}\u = 0$ on $\widetilde\D$ and $(E^{\ell,0}\u)_+=0$ on $\partial\widetilde\D\cap \partial\D$ by construction, we may deduce from \eqref{eq:outoing-bound}, with $\D$ replaced by $\widetilde \D$, that
\begin{alignat*}{5}
\|E^{\ell,0}\u\|_{L^p(\partial\D^\ell;|\sn|)}
&\stackrel{\eqref{eq:tilde2}}{=}\|(E^{\ell,0}\u)_+\|_{L^p(\partial\widetilde\D;|\sn|)} 
 \stackrel{\eqref{eq:outoing-bound}}{=}\|(E^{\ell,0}\u)_-\|_{L^p(\partial\widetilde\D;|\sn|)} \\
&\stackrel{\eqref{eq:tilde2}}{=}\|(E^{\ell,0}\u)_-\|_{L^p(\partial\widetilde\D \cap \partial\D;|\sn|)} 
 =\|\u_+\|_{L^p(\partial\D;|\sn|)}.
\end{alignat*} 
In the last step, we used the continuity of $E^{\ell,0}\u$ across $\partial\D$ and the fact that the normal vectors at $\partial\widetilde\R \cap \partial\R$ and $\partial\R$ have opposite sign. 
Using \eqref{eq:trace_u} and $(E^{\ell,\a} \u)(\r,\s)=0$ for all $(\r,\s) \in \partial\D^\ell$ with $\s \cdot \n(\r) \le \eta$, 
we obtain the second estimate of the lemma.
\end{proof}
\begin{remark} \rm
An important consequence of Theorem~\ref{thm:1} is that the trace of the solution $u^{\ell,\a}$ of the extended problem 
\eqref{eq:ext1}--\eqref{eq:ext2} is an element of $L^p(\partial\D^\ell)$ without weight, i.e., it has somewhat 
higher regularity. This is due to the geometric setting and the purely absorbing but non-scattering behaviour of the surrounding layer and will be important for our further considerations.
\end{remark}

\section{A modified boundary condition} \label{sec:reflection}

The estimate \eqref{eq:decay} implies that the solution $\ula$ can be made arbitrarily small at the outer boundary $\partial\D^\ell$ by choosing the parameters $a,\ell$ sufficiently large. A perturbation of the boundary condition at $\partial\D^\ell$ should, therefore, only have a minor effect.  
As approximation for \eqref{eq:ext1}--\eqref{eq:ext2}, we thus consider in this section the following problem with modified boundary conditions
\begin{alignat}{4}
\sgrad \w^{\ell,\a} + \mu^{\ell,\a} \w^{\ell,\a} &= \K^\ell \w^{\ell,\a} + q^\ell &\quad& \text{in } \D^\ell, \label{eq:per1}\\
\w^{\ell,\a}_- &= R \w^{\ell,\a}_+ && \text{on } \partial\D^\ell,  \label{eq:per2}
\end{alignat}
where $R : L^p(\partial\D^\ell;|\sn|) \to L^p(\partial\D^\ell;|\sn|)$ is an appropriate reflection operator.
Motivated by the considerations of Section~\ref{sec:variational}, we here consider the particular choice
\begin{align}
(R g)(\r,\s) 
= \frac{\sn+1}{\sn-1} g_+(\r,-\s), \qquad (\r,\s) \in \partial\D^\ell. \label{eq:R}
\end{align}
Particles arriving in direction $-\s$ at the boundary $\partial\D^\ell$, 
thus, partially leave the domain or get, otherwise, reflected in the opposite direction $\s$. 
In the analysis of this section, we will, only make use of the following properties.
\begin{lemma} \label{lem:R}
The operator $R : L^p(\partial\D^\ell;|\sn|) \to L^p(\partial\D^\ell;|\sn|)$ 
is linear and $R g =(Rg)_-$. Moreover, $|Rg(r,s)|\leq |g(r,-s)|$ for a.e. $(r,s)\in\partial\D_-^\ell$. If (A4) holds, then $|Rg(r,s)| \le (1-\eta) |g(r,-s)|$ if $\ell(\r,-\s)<\infty$.
\end{lemma}
\begin{proof}
The validity of the assertions follows directly from the definition \eqref{eq:R}. 
\end{proof}

\subsection{Well-posedness of the perturbed problem}

We will now show by a contraction argument that problem \eqref{eq:per1}--\eqref{eq:R} admits a unique solution. 
The key ingredient is that most particles that leave the domain $\R$ get absorbed before they arrive at the reflection boundary $\partial\R^\ell$. 
Moreover, points $(\r,\s) \in \D^\ell$ with $\ell(\r,\s)=\infty$ cannot be reached by particles originating from the domain $\R$.
Let us denote by
\begin{align}\label{eq:H}
H_-&=\{ h_-\in L^p(\partial \D^\ell;|\sn|): h_-(\r,\s)=0 \text{ if } \ell(\r,-\s)=\infty\}
\end{align}
the space of inflow boundary values at $\partial\D^\ell$ which corresponds to particles that may hit the 
computational domain $\R$ after travelling along straight lines through the extension layer $\R^\ell \setminus \R$. 
The following result is essential for our contraction argument.
\begin{lemma} \label{lem:z}
Let (A1)--(A4) hold. Then for any $q\! \in\! L^p(\D)$ and $h_- \!\!\in\!\! H_-$, the problem 
\begin{alignat}{4}
\sgrad \z^{\ell,\a} + \mu^{\ell,\a} \z^{\ell,\a} &= \K^\ell \z^{\ell,\a} + q^\ell &\quad& \text{in } \D^\ell,  \label{eq:aux1}\\
\z^{\ell,\a}_- &= h_- && \text{on } \partial\D^\ell, \label{eq:aux2}
\end{alignat}
has a unique solution $\z^{\ell,\a} \in W^p(\D^\ell)$ and $\z^{\ell,\a}_+(\r,\s)=0$ for a.e. point $(\r,\s) \in \partial\D^\ell$ with $\ell(\r,\s)=\infty$. Moreover, there holds
\begin{align*}
\|\z^{\ell,\a}\|_{W^p(\D)} 
\le C (\|q\|_{L^p(\D)} + e^{-\a \ell} \|h_-\|_{L^p(\partial\D^\ell;|\sn|)})
\end{align*}
with constant $C$ depending only on $\overline{\mu}$ and ${\rm diam}(\R)$.
In addition,
\begin{align*}
\|\z^{\ell,\a}_+\|_{L^p(\partial\D^\ell;|\sn|)}^p 
\le e^{-p\a \ell} \big(e^{-p\a\ell} \|h_-\|_{L^p(\partial\D^\ell;|\sn|)}^p + p \|q\|_{L^p(\D)} \|z^{\ell,\a}\|_{L^p(\D)}^{p-1} \big).
\end{align*}
\end{lemma}
\begin{proof}
Existence of a unique solution $z^{\ell,\a} \in W^p(\D^\ell)$ follows with the same arguments as in Theorem~\ref{thm:1a}.
Now let $\widetilde \D = \D^\ell \setminus \D$ denote the extension layer. 
Then due to the linearity of the problem, we can decompose $\z^{\ell,\a}$ on $\widetilde \D$ as
$\z^{\ell,\a} = \wz^\a + E^{\ell,\a} \z$, where $\z = \z^{\ell,\a}|_\D$ and $\wz^\a$ 
is the solution of the auxiliary problem 
\begin{align*}
\sgrad \wz^\a + \a \wz^a &= 0 \qquad \!\! \text{in } \widetilde \D, \\
\wz^\a_- &= h_-\quad \text{on } \partial\widetilde\D_- \cap \partial\D^\ell
\quad \text{and} \quad
\wz^\a_-= 0 \quad \text{on } \partial\widetilde\D_- \cap \partial\D.
\end{align*}
With similar arguments as in the proof of Lemma~\ref{lem:extension}, one can show that
\begin{align*}
\|\wz^\a\|_{L^p{(\partial\D;|\sn|)}}
&= \|\wz^\a_+\|_{L^p{(\partial\widetilde\D \cap \partial\D;|\sn|)}} \\
&\le e^{-\a \ell}\|\wz^0_+\|_{L^p(\partial\widetilde\D \cap \partial\D;|\sn|)}
 \le e^{-\a \ell}\|h_-\|_{L^p(\partial\D^\ell;|\sn|)}.
\end{align*}
In the last step, we used the a-priori estimate of Theorem~\ref{thm:1a} for the problem defining the solution $\wz^a$ with $a=0$.
From the decomposition $\z^{\ell,\a} = \wz^a + E^{\ell,\a} \z$, the definition of $\z=\z^{\ell,\a}|_\D$, 
and the continuity of $\z^{\ell,\a}$ across $\partial\D$, we deduce that
\begin{align*}
\z_- = \wz^\a_+ \qquad \text{on }  \partial\widetilde \D \cap \partial\D,
\end{align*}
i.e., the particles entering $\D$ via $\partial\D$ are those generated by $h_-$ on $\partial\D^\ell$ and leaving 
the surrounding layer $\widetilde\D = \D^\ell \setminus \D$ via $\partial\D$. 
The function $\z=\z^{\ell,\a}|_\D$ hence solves
\begin{alignat*}{4}
\sgrad \z + \mu \z &= \K \z + q &\quad& \text{in } \D, \\
\z_- &= g_-                    &&  \text{on } \partial\D,
\end{alignat*}
with boundary data $g_-=\wz^\a_+$. From Theorem~\ref{thm:1a} and the previous estimates, we get
\begin{align*}
\|\z^{\ell,\a}\|_{W^p(\D)} = \|\z\|_{W^p(\D)} \le C' (\|q\|_{L^p(\D)} + e^{-\a \ell}\|h_-\|_{L^p(\partial\D^\ell;|\sn|)}).   
\end{align*}
The additional bound for the outgoing trace $\z^{\ell,\a}_+=E^{\ell,a}z$ can then be deduced from the second estimate of Theorem~\ref{thm:1a}
and Lemma~\ref{lem:extension}. 
\end{proof}
We are now in the position to establish the well-posedness of problem~\eqref{eq:per1}--\eqref{eq:R}.
\begin{theorem}\label{thm:2}
Let (A1)--(A4) hold. Then, for any $q \in L^p(\D)$ and any $a > 0$, 
problem \eqref{eq:per1}--\eqref{eq:per2} has a unique solution $\w^{\ell,\a} \in W^p(\D^\ell)$ with
\begin{align*}
\|\w^{\ell,\a}\|_{W^p(\D)} \leq C \|q\|_{L^p(\D)} \quad\text{and}\quad\|\w^{\ell,\a}\|_{L^p(\partial\D^\ell)} \le C e^{-\a \ell} \|q\|_{L^p(\D)}
\end{align*}
with constant $C$ depending only on $\overline{\mu}$, ${\rm diam}(\R)$ and $\eta$.
\end{theorem}
\begin{proof}
In a first step, we show that for any solution $\wla \in W^p(\D^\ell)$ there holds
\begin{align} \label{eq:star}
\wla(\r,\s)=0 \qquad \text{for a.e. } (\r,\s) \in \partial\D^\ell \text{ with } \ell(\r,\s)=\ell(\r,-\s)=\infty,
\end{align}
which implies that the $\wla \equiv 0$ on the union of all lines that do not intersect the computational domain $\R$.
which implies that the $\wla \equiv 0$ on the union of all lines that do not intersect the computational domain $\R$; cf. Remark~\ref{rem:geometry}.
Let $(\r,\s)$ be such a point on the outer boundary $\partial\D^\ell_-$ with $\ell(\r,\s)=\ell(\r,-\s)=\infty$. 
Then $\r+t \s \in \R^\ell \setminus \R$ for all $0 < t < t^*$, where $t^*$ is chosen such that $\r^*=\r+t^*\s \in \partial\R^\ell$; see Figure~\ref{fig:sketch}. Since the medium in the extension layer $\R^\ell \setminus \R$ is purely absorbing, we have $\wla(\r+t\s,\s) = e^{-at} \wla(\r,\s)$. Applying the reflection operator at the point $(\r^*,\s)$, we further obtain 
\begin{align*}
|\wla(\r^*,-s)|=|R\wla(\r^*,-\s)| \le |\wla(r^*,\s)|=e^{-at^*} |\wla(\r,\s)|. 
\end{align*}
Repeating the argument with $\r = \r^*-t^*\s$ yields $|\wla(\r,\s)|\leq e^{-2a t^*} |\wla(\r,\s)|$, which implies that $\wla(\r,\s)=0$ and shows the assertion \eqref{eq:star}. 

As a consequence, we know that any solution $\wla \in W^p(\D)$ has to satisfy $\wla|_{\partial\D^\ell_-} = h_- \in H_-$ as in \eqref{eq:H}. 
We now show the existence and uniqueness of such a solution. 
For any given $h_- \in H_-$, we define $\Phi(h_-):=R \z_+$, where $\z\in W^p(\D^\ell)$ is the unique solution of 
\begin{alignat}{4}
\sgrad \z + \mu^{\ell,\a} \z &= \K^\ell z + q^\ell &\quad& \text{in } \D^\ell, \label{eq:auxil1}\\
\z_- &= h_- && \text{on } \partial\D^\ell. \label{eq:auxil2}
\end{alignat}
The results of Lemma~\ref{lem:z} imply that $z_+(\r,\s)=0$ for $(\r,\s) \in \partial\D^\ell$ with $\ell(\r,\s)=\infty$, 
and thus $Rz_+\in H_-$ by Lemma~\ref{lem:R}. 
Hence $\Phi:H_- \to H_-$ is a self-mapping on the non-empty and closed subset $H_-$ of the Banach-space $L^p(\partial\D^\ell;|\sn|)$.
By taking the difference of two solutions $z,z'$ with boundary data $h_-, h'_- \in H_-$, 
we further deduce from Lemma~\ref{lem:R} and Lemma~\ref{lem:z} that
\begin{align*}
\|\Phi(h_-) &- \Phi(h'_-)\|_{L^p(\partial\D^\ell;|\sn|)} =\|R z_+ - R z'_+\| \\
&\le (1-\eta) \| z_+ -z'_+\|_{L^p(\partial \D^\ell;|\sn|)} 
 \le  (1-\eta) e^{- 2\a \ell} \| h_- - h'_-\|_{L^p(\partial \D;|\sn|)}.
\end{align*}
This shows that $\Phi$ is a contraction on $H_-$ and by Banach's fixed-point theorem, 
there exists a unique fixed point $h_-\in H_-$ with $\Phi(h_-)=h_-$.
By construction, the function $w^{\ell,\a}=\z$ with $\z$ as defined above then is the unique solution of
\eqref{eq:per1}--\eqref{eq:per2}.
Now set $h_-^0=0$ and for $n\ge 1$ define $h_-^n = \Phi(h_-^{n-1})$. 
Then from the convergence estimates for Banach's fixed-point iteration, we obtain
\begin{align*}
\|h_-\|_{L^p(\partial\D^\ell;|\sn|)} 
&=\|h_--h_-^0\|_{L^p(\partial\D^\ell;|\sn|)} 
\le \frac{1}{1-\eta} \|h_-^1-h_-^0\|_{L^p(\partial\D^\ell;|\sn|)}.
\end{align*}
Due to the choice $h_-^0=0$, we know that $h_-^1 = R \ula$, where $\ula$ is the unique solution of \eqref{eq:ext1}--\eqref{eq:ext2}. 
From Lemma~\ref{lem:R} and the estimate \eqref{eq:decay}, we can then deduce that
\begin{align*}
\eta^{1/p}\|h_-\|_{L^p(\partial\D^\ell)} & \le \|h_-\|_{L^p(\partial\D^\ell;|\sn|)} \\ 
& \le \frac{1}{1-\eta} \|\ula_+\|_{L^p(\partial\D^\ell;|\sn|)} \le  \frac{C}{1-\eta} e^{-a\ell}\|q\|_{L^p(\D)}.
\end{align*}
From the construction of $h_-$, one can see that the solution $w^{\ell,\a}=z$  of the auxiliary problem \eqref{eq:auxil1}--\eqref{eq:auxil2} is the unique solution of problem \eqref{eq:per1}--\eqref{eq:per2}. The proof is thus completed by an application of Lemma~\ref{lem:z}, which yields the bounds for the solution. 
\end{proof}

\subsection{Error estimates}\label{sec:error_estimates}

In preparation of the next theorem, let us state a bound for the solution $\wla$ of 
the perturbed problem on the extension layer.
\begin{lemma}\label{lem:cor1}
Let (A1)--(A4) hold and $w^{\ell,\a}$ be the solution of \eqref{eq:per1}--\eqref{eq:per2}. Then 
\begin{align*}
\|\sgrad w^{\ell,\a}\|_{L^p(\D^\ell\setminus\D)} + \a \|w^{\ell,a}\|_{L^p(\D^\ell\setminus\D)} \leq C \a^{\frac{p-1}{p}} \|q\|_{L^p(\D)},
\end{align*}
with constant $C$ depending only on $\bar \mu$, ${\rm diam}(\R)$, and $\eta$. \end{lemma}
\begin{proof}
Observe that $w^{\ell,a}$ is a solution to
\begin{alignat*}{4}
\sgrad w^{\ell,a} + \a w^{\ell,a} &= 0  &\quad &\text{in } \widetilde \D, \\
w^{\ell,a}_- &=  (w^{\ell,a}|_{\D})_+  &\quad &\text{on } \partial\widetilde\D \cap \partial\D,\\
w^{\ell,a}_- &= R w^{\ell,a}  &\quad & \text{on } \partial\widetilde\D \cap \partial\D^\ell.
\end{alignat*}
In view of Theorem~\ref{thm:2}, we already know that 
\begin{align*}
\|w^{\ell,a}_+\|_{L^p(\partial\D;|\sn|)}&\leq C \|w^{\ell,a}\|_{W^p(\D)}\leq C \|q\|_{L^p(\D)}, \qquad \text{and}\\
\|R w^{\ell,a}\|_{L^p(\D^\ell;|\sn|)} &\leq \| w^{\ell,a}\|_{L^p(\D^\ell;|\sn|)}\leq C e^{-a\ell}\|q\|_{L^p(\D)}.
\end{align*}	
The assertion now follows with the same arguments as in the proof of Lemma~\ref{lem:extension}.
\end{proof}
In combination with the previous results, we can now derive explicit estimates for the perturbation error 
resulting from the use of the reflection boundary condition. 
\begin{theorem} \label{thm:2a}
Let (A1)--(A4) hold and let $u$ and $w^{\ell,\a}$ denote the solutions of problem \eqref{eq:rte1}--\eqref{eq:rte2} and of problem \eqref{eq:per1}--\eqref{eq:per2}, respectively. Then 
\begin{align*}
\|w^{\ell,\a} - \u\|_{W^p(\D)}\le C e^{-2\a \ell} \|q\|_{L^p(\D)},
\end{align*}
with constant $C$ depending only on $\bar \mu$, ${\rm diam}(\R)$, and $\eta$. 
Moreover,
\begin{align*}
\|\sgrad (w^{\ell,\a} - E^{\ell,a}\u)\|_{L^p(\D^\ell\setminus\D)} + a\|(w^{\ell,\a} - E^{\ell,a}\u)\|_{L^p(\D^\ell\setminus\D)} \leq C a^{\frac{p-1}{p}} e^{-a\ell} \|q\|_{L^p(\D)}.
\end{align*}
\end{theorem}
\begin{proof}
By Theorem~\ref{thm:1}, we have $\u=\u^{\ell,\a}|_\D$, where $\u^{\ell,\a}$ is the solution of 
\eqref{eq:ext1}--\eqref{eq:ext2}. The difference $\z^{\ell,\a}=\w^{\ell,\a}-u^{\ell,\a}$ satisfies \eqref{eq:aux1}--\eqref{eq:aux2} with $h_- = R w^{\ell,\a}$ and $q^\ell=0$. The first bound then follows by a combination of Lemma~\ref{lem:R}, Lemma~\ref{lem:z}, and Theorem~\ref{thm:2}, and the second estimate follows similarly using Lemma~\ref{lem:cor1}.
\end{proof}
\bigskip 

\begin{center}
\sc Part 2: Numerical approximation
\end{center}

\medskip 

In the following two sections, we discuss the numerical approximation of problem~\eqref{eq:per1}--\eqref{eq:per2}
by extending the mixed variational approach proposed in \cite{EggerSchlottbom12}. 
We first derive a variational formulation of the problem and consider
its systematic Galerkin approximation, and then discuss a particular method
based on a tensor product approximation using spherical harmonics and mixed finite elements. 
\section{A mixed variational problem} \label{sec:variational}
For ease of notation, we write 
$w=w^{\ell,a}$ and $q=q^\ell$ in the following and consider the problem
\begin{alignat}{4}
\sgrad \w + \mu^{\ell,a} \w &= \K^\ell \w + q &\quad & \text{in } \D^\ell, \label{eq:sys1}\\
\w_- &= R \w_+ && \text{on } \partial\D^\ell.  \label{eq:sys2}
\end{alignat}
As before, the reflection operator is defined by $(R g)(\r,\s) = \frac{\sn+1}{\sn-1} g_+(\r,-\s)$,
and the particular form will become important now. 
Based on the derivation of the perturbed problem, we know that $q \equiv 0$ and $\k^\ell \equiv 0$ in the extension layer $\partial\R^\ell \setminus \R$.
\subsection{Even-odd splitting}
Following \cite{EggerSchlottbom12}, we start with splitting functions $v(\r,\s)$ into even and odd parts with respect to direction $\s$ defined by 
\begin{align} \label{eq:split}
v^\pm(\r,\s) = \frac{1}{2} \left( v(\r,\s) \pm \v(\r,-\s)\right).
\end{align}
Let us note that the splitting $v=v^++v^-$ is orthogonal with respect to the scalar product of $L^2(\S)$. This allows us to rewrite the problem \eqref{eq:sys1}--\eqref{eq:sys2} as follows.
\begin{lemma}
Let $w \in W^2(\D^\ell)$ denote a solution of problem \eqref{eq:sys1}--\eqref{eq:sys2}. Then 
\begin{alignat}{4}
\sgrad w^- + \mu^{\ell,a} w^+ &= \K^\ell w^+  + q^+ &\quad& \text{in } \D^\ell, \label{eq:split1}\\
\sgrad w^+ + \mu^{\ell,a} w^- &= \K^\ell w^-  + q^- &\quad& \text{in } \D^\ell, \label{eq:split2}\\
w^+ &= \sn w^- && \text{on } \partial\D. \label{eq:split3}
\end{alignat}
If, on the other hand, $w^\pm \in W^2(\D^\ell)$ solve \eqref{eq:split1}--\eqref{eq:split3}, then $w=w^++w^- \in W^2(\D^\ell)$ is a solution of \eqref{eq:sys1}--\eqref{eq:sys2}. 
The two problems are thus equivalent in this sense.
\end{lemma}
\begin{proof}
Let us note that multiplication with $\mu$ and application of $\K^\ell$ preserves parity, i.e., these operations map even to even and odd to odd functions, while application of 
$\sgrad$ reverts the parity. Together with the orthogonality of the splitting \eqref{eq:split} this already shows the equivalence of \eqref{eq:sys1} and \eqref{eq:split1}--\eqref{eq:split2}. 
Using the definition of the reflection operator, the boundary condition \eqref{eq:sys2} can be rewritten as
\begin{align*}
(1-\sn ) w(\r,\s) = -(1+\sn) w(\r,-\s) \qquad \text{for } (\r,\s) \in \partial\D^\ell_-.
\end{align*}
A reordering of the terms and inserting the definition of $w^\pm$ further yields 
\begin{align*}
2 w^+(\r,\s) 
&=w(\r,\s)+w(\r,-\s) \\
&= \sn [w(\r,\s)-w(\r,-\s)] = 2 \sn w^{-}(\r,\s) \qquad \text{for } (\r,\s) \in \partial\D^\ell_-,
\end{align*}
which shows that \eqref{eq:split3} is valid on $\partial\D_-$. 
Now note that the left and right hand side of the last identity each define even functions of $s$. This shows that \eqref{eq:split3} also holds on $\partial\D^\ell_+$. The equivalence of \eqref{eq:sys2} with \eqref{eq:split3} follows by reverting the arguments.
\end{proof}
\subsection{Variational characterization}
We can now use the equivalent formulation \eqref{eq:split1}--\eqref{eq:split3} to derive a weak 
form of problem \eqref{eq:sys1}--\eqref{eq:sys2}. The function spaces
\begin{align*}
\WW^+ = \{u^+ \in W^2(\D^\ell) : u^+|_{\partial\D^\ell} \in L^2(\partial\D^\ell)\} 
\qquad \text{and} \qquad
\VV^\pm = \{u^\pm \in L^2(\D^\ell)\}
\end{align*}
turn out to be appropriate for representing the even and odd solution components of the problem under investigation.
The tensor product space $\WW^+ \times \VV^-$ is equipped with its natural norm given by
\begin{align*}
\tnorm{(u^+,u^-)}^2 
= \|\sgrad u^+\|^2_{L^2(\D^\ell)} + \|u^+\|_{L^2(\partial\D^\ell)} + \|u^+\|^2_{L^2(\D^\ell)}+ \|u^-\|^2_{L^2(\D^\ell)}. 
\end{align*}
For ease of notation, we further define the total collision operator 
\begin{align*}
\C : L^2(\D^\ell) \to L^2(\D^\ell), \quad v \mapsto \mu^{\ell,a} v - \K^\ell v. 
\end{align*}
We then obtain the following variational characterization of solutions.
\begin{lemma} \label{lem:variational}
Let $w \in W^2(\D^\ell)$ denote a solution of problem \eqref{eq:sys1}--\eqref{eq:sys2} or, equivalently, of problem \eqref{eq:split1}--\eqref{eq:split3}. Then for all $v^+ \in \WW^+$ and $v^- \in \VV^-$ there holds 
\begin{align}
(w^+,v^+)_{\partial\D^\ell} + (\C w^+,v^+)_{\D^\ell} - (w^-, \sgrad v^+)_{\D^\ell} &= (q^+,v^+)_{\D^\ell}, \label{eq:var1} \\
(\sgrad w^+,v^-)_{\D^\ell} + (\C w^-,v^-)_{\D^\ell} &= (q^-,v^-)_{\D^\ell}.\label{eq:var2}
\end{align}
\end{lemma}
\begin{proof}
Recall that $(u,v)_M = \int_M u v dM$ denotes the scalar product of $L^2(M)$.
Multiplying \eqref{eq:split1} with a test function $v^+ \in \WW^+$ and integrating over $\D^\ell$ yields 
\begin{align*} 
(q^+ -\C w^+,v^+)_{\D^\ell}  
&= (\sgrad w^-,v^+)_{\D^\ell} 
 = -(w^-,\sgrad v^+)_{\D^\ell} + (\sn w^-,v^+)_{\partial\D^\ell}.
\end{align*}
Here, we made use of the integration-by-parts formula \eqref{eq:ipp} in the last step. 
The boundary condition \eqref{eq:split3} allows us to replace the last term, and inserting the definition of the collision operator $\C$ then already yields \eqref{eq:var1}. 
The validity of equation \eqref{eq:var2} follows immediately by testing \eqref{eq:split2} with $v^- \in \VV^-$.
\end{proof}
Let us note at this point that, due to the particular reflection boundary condition, no half-space integrals appear in the variational characterization of the perturbed problem. 

\subsection{Weak formulation}
We can now give the following weak formulation of problem \eqref{eq:sys1}--\eqref{eq:sys2} and of the equivalent problem \eqref{eq:split1}--\eqref{eq:split3}, respectively.
\begin{problem}\label{prob:weak}
Find $w^+ \in \WW^+$ and $w^- \in \VV^-$ such that \eqref{eq:var1}--\eqref{eq:var2} holds.  
\end{problem}
Let us note that existence of a weak solution is immediately obtained from Theorem~\ref{thm:2} and Lemma~\ref{lem:variational}. To show uniqueness and to facilitate the further discussion, we will assume in the following that
\begin{myass}
\item[(A5)] $\gamma \|v\|_{L^2(\D^\ell)}^2 \le (\C v, v)_{\D^\ell} \le \Gamma \|v\|_{L^2(\D^\ell)}$ for all $v\in L^2(\D^\ell)$ for some $0<\gamma,\Gamma$.
\end{myass}
This condition is valid, e.g., if the medium is uniformly absorbing and it implies that the artificial absorption has to satisfy $\gamma \le \a \le \Gamma$ as well. Using the arguments of \cite[Section 3.3]{EggerSchlottbom12}, the assumption could be further relaxed. 
Due to (A5) the total collision operator $\C: L^2(\D^\ell) \to L^2(\D^\ell)$ is boundedly invertible, which allows to define norms
\begin{align*}
\|u\|_{\C}^2 = (\C u, u)_{\D^\ell} \qquad \text{and} \qquad \|u\|_{\Ci}^2 = (\Ci u, u)_{\D^\ell},
\end{align*}
which are equivalent to the norm on $L^2(\D^\ell)$, i.e., 
$\gamma \|u\|_{L^2(\D^\ell)}^2 \le \|u\|_{\C}^2 \le \Gamma \|u\|_{L^2(\D^\ell)}^2$ and 
$\Gamma^{-1} \|u\|_{L^2(\D^\ell)}^2 \le \|u\|_{\Ci}^2 \le \gamma^{-1} \|u\|_{L^2(\D^\ell)}^2$.  
By minor modification of the arguments used in \cite{EggerSchlottbom12}, we can now deduce the following assertion.
\begin{theorem}\label{thm:3a}
Let (A1)--(A5) hold. Then problem \eqref{eq:var1}--\eqref{eq:var2} has a unique solution $w^+ \in \WW^+$ and $w^- \in \VV^-$ and there holds
\begin{align*}
\tnorm{(w^+,w^-)} \le C_D \|q\|_{L^2(\D)},
\end{align*}
with constant $C_D$ depending at most linearly on $\gamma^{-1}$ and $\Gamma$.
In addition, the function $w=w^++ w^-\in W^2(\D^\ell)$ coincides with the unique solution of \eqref{eq:sys1}--\eqref{eq:sys2}. 
\end{theorem}
\begin{proof}
The proof of \cite[Theorem 3.1]{EggerSchlottbom12} applies almost verbatim and yields the existence and uniqueness of a solution as well as the a-priori estimate
\begin{align*}
\|\sgrad u^+\|^2_{\Ci} + \|u^+\|^2_{\C} + \|u^+\|_{L^2(\partial\D^\ell)}^2 + \|u^-\|_{\C}^2
\le C \|q\|_{\Ci}^2 
\end{align*}
with a universal constant $C>0$. Let us note that the change of the boundary term does not affect the proof given in \cite{EggerSchlottbom12}.
The postulated bounds for the solution then follow using assumption (A5) and the equivalence of the norms $\|\cdot\|_\C$, $\|\cdot\|_{\Ci}$, and $\|\cdot\|_{L^2(\D^\ell)}$. 
The last assertion follows by the uniqueness of the weak solution and noting that, according to Lemma~\ref{lem:variational}, the solution of \eqref{eq:sys1}--\eqref{eq:sys2} also solves \eqref{eq:var1}--\eqref{eq:var2}.
\end{proof}
\subsection{Galerkin approximation}
Let $\WW_h^+ \subset \WW^+$ and $\VV_h^- \subset \VV^-$ be closed subspaces. We then consider the following Galerkin approximation of Problem~\ref{prob:weak}.
\begin{problem}\label{prob:galerkin}
Find $(w_h^+,w_h^-) \in \WW_h^+ \times \VV_h^-$ such that 
\begin{align}
(\C w_h^+,v_h^+)_{\D^\ell} + (w_h^+,v_h^+)_{\partial\D^\ell} - (w_h^-, \sgrad v_h^+)_{\D^\ell} &= (q^+,v_h^+)_{\D^\ell} && \forall v_h^+ \in \WW_h^+, \label{eq:gal1} \\
(\sgrad w_h^+,v_h^-)_{\D^\ell} + (\C w_h^-,v_h^-)_{\D^\ell} &= (q^-,v_h^-)_{\D^\ell} && \forall v_h^- \in \VV_h^-.\label{eq:gal2}
\end{align}
\end{problem}
In order to ensure the existence of a unique discrete solution, we require that 
\begin{myass}
\item[(A6)] $\WW_h^+ \subset \WW^+$, $\VV_h^- \subset \VV^-$ are finite dimensional and $\{\sgrad w_h^+: w_h^+\in \WW_h^+\}\subset \VV_h^-$.
\end{myass}
This condition guarantees the uniform stability of the discrete variational problem. 
By the same arguments as used in \cite[Section 6]{EggerSchlottbom12}, we then obtain the following results.
\begin{lemma}\label{lem:approx_error}
Let assumptions (A1)--(A6) be valid. 
Then Problem~\ref{prob:galerkin} has a unique solution $(w_h^+,w_h^-) \in \WW_h^+ \times \VV_h^-$ and 
\begin{align*}
\tnorm{(w_h^+,w_h^-)} \le C_D \|q\|_{L^2(\D)},
\end{align*}
with the same constant $C_D$ as in Theorem~\ref{thm:3a}. Moreover,
\begin{align*}
\tnorm{(w^+ - w_h^+,w^--w_h^-)} \le C_D' \inf \tnorm{(w^+ - v_h^+,w^- - v_h^-)},
\end{align*}
where the infimum is taken over all $(v_h^+,v_h^-)\in \WW_h^+ \times \VV_h^-$. The constant $C_D'$ again depends at most linearly on $\gamma^{-1}$ and $\Gamma$.
\end{lemma}
\begin{proof}
The assertions result from application of the Babuska-Aziz lemma. 
Details can be found in the proof of \cite[Theorem~6.1]{EggerSchlottbom12}. 
\end{proof}
Together with the results of Section~\ref{sec:reflection},
we finally obtain the following error estimate.
\begin{theorem}\label{thm:5}
Let (A1)--(A6) hold and let $u$, $w=w^{\ell,a}$, and $(w_h^+,w_h^-)$ denote the unique solutions of \eqref{eq:rte1}--\eqref{eq:rte2}, of \eqref{eq:sys1}--\eqref{eq:sys2}, and of Problem~\ref{prob:galerkin}, respectively. Then
\begin{align*}
\| u^+ - w_h^+\|_{W^2(\D)}&+\|u^- - w_h^-\|_{L^2(\D^\ell)} \\
&\le C_D e^{-\a \ell}\|q\|_{L^2(\D)} 
+ C_D' \inf \tnorm{(w^+ - v_h^+,w^- - v_h^-)}.
\end{align*}
The infimum is again taken over all $(v_h^+,v_h^-)\in \WW_h^+ \times \VV_h^-$, and the constants $C_D$ and $C_D'$ depend at most linearly on $\gamma^{-1}$ and $\Gamma$.
\end{theorem}
\begin{proof}
The result follows from the previous results via the triangle inequality.
\end{proof}
\begin{remark} \rm
Assume that the best-approximation error can be bounded uniformly by $\inf\tnorm{(w^+ - v_h^+,w^- - v_h^-)}=O(h)$.
Then the optimal choice of the parameters $\ell,a$ would be such that $e^{-\ell a} \approx h$, 
where we neglect the at most linear dependence of $C_D'$ on $a$. 
Hence, it suffices to choose $a$ or $\ell$ proportional to $|\log h|$ in order to obtain 
a quasi-optimal overall approximation.
Moreover, in view of Theorem~\ref{thm:2a}, $\inf\tnorm{(w^+ - v_h^+,w^- - v_h^-)}$ can replaced by $\inf\tnorm{(E^{\ell,a}u^+ - v_h^+,E^{\ell,a}u^- - v_h^-)}$ in the estimate of Theorem~\ref{thm:5}, which can be made explicit under regularity assumptions on the solution $u$ of the original problem.
\end{remark}

\section{The $P_N$-finite element method} \label{sec:pnfem}
We now discuss a particular construction of approximation spaces $\WW_h^+$ and $\VV_h^-$
using spherical harmonics and finite elements.
\subsection{Angular approximation}
As angular basis functions $H_n$ in the moment expansion \eqref{eq:fourier}, 
we employ the spherical harmonics $Y_l^m$, $-l \le m \le l$, $l \ge 0$ in the sequel.
These functions form an orthonormal basis for $L^2(\S)$ and allow to efficiently realize the splitting \eqref{eq:split}, since $Y_{2l}^m$ and $Y_{2l+1}^m$ are even and odd functions, respectively.
For the approximation of even and odd functions of angular variable $\s$, we then consider the spaces
\begin{align*}
\SS_N^+ &= {\rm span}\{Y_{2l}^m, 0\leq 2l\leq N, -2l\leq m\leq 2l\},\\
\SS_N^- &= {\rm span}\{Y_{2l+1}^m, 0\leq 2l+1\leq N, -2l-1\leq m\leq 2l+1\}.
\end{align*}
Let us note that ${\rm dim\,} (\SS_N^\pm) \approx N^{2}$. 
We will later only consider the choice $N$ odd; necessary modifications for the case $N$ even can be found in \cite{EggerSchlottbom12}.

\subsection{Spatial approximation} 
We denote by $\T_h=\T_h(\R^\ell)$ a quasi-uniform regular partition of the spatial domains $\R^\ell \subset \RR^3$ into simplicial elements $T$ of size $h$
and assume that $\T_h(\R) = \{T \in \T_h(\R^\ell) : T \subset \overline\R\}$ is 
a conforming mesh of the original domain $\R \subset\R^\ell$.
By $\P_k(\T_h) = \{v \in L^2(\R^\ell) : v|_T \in P_k(T)\}$, we denote the spaces of piecewise polynomials over $\T_h$ of degree less or equal to $k$. For the approximation of the even and odd moments $u_n$ in the expansion \eqref{eq:fourier}, we utilize the spaces
\begin{align*}
\XX_h^+ =\P_{1}(\T_h)\cap H^1(\R^\ell) \qquad \text{and} \qquad 
\XX_h^- =\P_{0}(\T_h)\subset L^2(\R^\ell).
\end{align*}
We denote by $\{\varphi_j\}$ and $\{\chi_k\}$ the canonical basis consisting of hat functions and piecewise constant functions, respectively, and recall that ${\rm dim}(\XX_h^+) \approx {\rm dim}(\XX_h^-) \approx h^{-3}$.

\subsection{Tensor product spaces}
As choice for the spaces $\WW_h^+$ and $\VV_h^-$ in Problem~\ref{prob:galerkin}, we then consider the following tensor product construction
\begin{align*}
\WW_h^+= \XX_h^+ \otimes \SS_N^+ \quad\text{and}\quad
\VV_h^-= \XX_h^- \otimes \SS_N^-.
\end{align*}
By similar arguments as in \cite{EggerSchlottbom12,WrightArridgeSchweiger09}, one can show the following properties.
\begin{lemma}
Let $\WW_h^+$ and $\VV_h^-$ be as above. 
Then ${\rm dim}(\WW_h^+) \approx {\rm dim}(\VV_h^-) \approx h^{-3} N^{2}$. If $N$ is chosen odd, then assumption (A6) is satisfied.
\end{lemma}
\begin{proof}
The estimates for the dimensions are obtained directly from the tensor product construction.
By well-known recurrence relations for spherical harmonics, one can further see that $s Y_l^m$ is of the form \cite{Arridge99,WrightArridgeSchweiger09} 
$$
s Y_l^m = 
\begin{pmatrix} 
a_{1lm} Y_{l-1}^{m-1} + b_{2lm} Y_{l-1}^{m+1} + c_{3lm} Y_{l+1}^{m-1} + d_{4lm} Y_{l+1}^{m+1} \\
a_{1lm} Y_{l-1}^{m-1} + b_{2lm} Y_{l-1}^{m+1} + c_{3lm} Y_{l+1}^{m-1} + d_{4lm} Y_{l+1}^{m+1} \\
e_{lm} Y_{l-1}^{m} + f_{lm} Y_{l+1}^{m}
\end{pmatrix}.
$$
Together with the assumption that $N$ is odd, this implies that $s \SS_N^+ \subset (\SS_N^-)^3$. 
Since the derivative of a continuous piecewise linear function is piecewise constant, we further have $\nabla \XX_h^+ \subset (\XX_h^-)^3$.  
The compatibility condition in (A6) then follows directly from the tensor product construction.  
\end{proof}
As a consequence of the previous lemma, all results presented in Section~\ref{sec:variational} 
apply directly to the $P_N$-finite element method based on this choice of approximation spaces.

\subsection{Complexity estimates}

The choice of a basis for $\WW_h^+$ and $\VV_h^-$ allows to recast the discrete variational 
problem \eqref{eq:gal1}--\eqref{eq:gal2} as a linear system 
\begin{alignat*}{2}
\ttM \ttup + \ttR \ttup - \ttB^\top \ttum &= \ttqp, \\
\ttB \ttup + \ttC \ttum              &= \ttqm,
\end{alignat*}
where $\ttup,\, \ttum,\, \ttqp,\, \ttqm$ are the corresponding coefficient vectors.
When choosing the natural tensor product basis with components $\{\varphi_j Y_{2m}^l\}$ for the even components and $\{\chi_j Y_{2m+1}^l\}$ for the odd components, the resulting system matrices can be seen to have some favourable properties. 
\begin{lemma} \label{lem:prop}
Let assumptions (A2)--(A3) and (A5)--(A6) hold and $\a > 0$. 
Then the matrices $\ttM$ and $\ttC$ are symmetric and positive definite and $\ttC$ is symmetric and positive semi-definite. 
If the basis for $\WW_h^+$ and $\VV_h^-$ are chosen as described above, then 
$\ttM$ and $\ttR$ are block-diagonal with sparse blocks, $\ttB$ is block-sparse with sparse blocks, and $\ttC$ is diagonal. 
Moreover, the number of non-zero entries is given by
$nnz(\ttM) \approx nnz(\ttB) \approx nnz(\ttC) \approx h^{-3} N^2$ and $nnz(\ttR) \approx h^{-2} N^2$. 
\end{lemma}
As a direct consequence of these properties, the multiplication with any of the system matrices can be achieved in order optimal complexity. 
\begin{remark} \rm 
In our numerical tests, we consider test problems with invariance in one  spatial direction. It then suffices to 
consider a two-dimensional cross-section $\R \subset \RR^2$ of the three-dimensional domain while the angular domain still is $\S=\SS^2$. In that case, the spatial mesh $\T_h$ consists of triangles and ${\rm dim}(\XX_h^\pm) \approx h^{-2}$ and consequently ${\rm dim}(\WW_h^+) \approx {\rm dim}(\VV_h^-) \approx h^{-2} N^2$. All observations made above apply with obvious modifications also to this setting.
\end{remark}

\subsection{Solution of the linear system} \label{sec:linsys}

Let us finally also comment briefly on the efficient solution of the linear system arising from the tensor-product $P_N$-finite element approximation. Since the matrix $\ttC$ is diagonal and positive definite, one can eliminate $\ttum$ via the second equation by
\begin{align} \label{eq:schur1}
\ttum = \ttC^{-1} (\ttqm - \ttB \ttup).
\end{align}
Note that $\ttum$ can be computed efficiently, once the even component $\ttup$ of the solution is known. 
Inserting the formula for $\ttum$ into the first equation yields the Schur complement system 
\begin{align} \label{eq:schur2}
[\ttM + \ttR + \ttB^\top \ttC^{-1} \ttB] \ttup = \ttqp + \ttB^\top \ttC^{-1} \ttqm.
\end{align}
Using Lemma~\ref{lem:prop}, the matrix $\ttS = [\ttM+\ttR+\ttB^\top \ttC^{-1} \ttB]$ is symmetric and positive definite. Moreover, the matrix vector product $\ttS \cdot \ttup$ can be realized, even without assembling $\ttS$, with $h^{-3} N^2$ algebraic operations and thus in optimal complexity.
For the efficient numerical solution of the Schur complement system, we can employ a preconditioned conjugate gradient (PCG) method.
In our numerical tests, we utilize a spatial multigrid strategy for preconditioning, cf. \cite{ArridgeEggerSchlottbom13,ChangLee2003}.

\section{Numerical illustrations} \label{sec:num}

We now illustrate the theoretical results obtained in the previous sections by some numerical tests.
\subsection{Example 1: Constant coefficients}
We choose the unit ball $\R=B_1(0)\subset \RR^2$ as computational domain, and 
define the model parameters as
\begin{align*}
k(r,s\cdot s')=\frac{10}{4\pi},\quad \quad \mu(r)=10+\frac{1}{10}.
\end{align*}
This corresponds to a scattering dominated regime with small absorption. 
Particles are introduced into the domain via an isotropic source
\begin{align*}
q(r,s)= \exp(-5 |\r - \r_0|^2), \quad (\r,\s)\in\D, \qquad \r_0=(\tfrac{3}{4},0).
\end{align*}
The layer domain is chosen as $\R^\ell=B_{6/5}(0)$, i.e., $\ell = 1/5$. For our numerical experiments we have chosen different extensions of the absorption coefficient, see for instance Table~\ref{tab:smooth_err_loc}.
Since an analytical solution is not available, we chose as a reference solution $\widehat u$ the $P_N$-FEM approximation $w^{\ell,a}_h$ with $\exp(-a\ell)=1/32$ computed for $N=11$ on a spatial grid with $177\,761$ vertices, which corresponds to $h=0.005$. This amounts to a total number of degrees of freedom of ${\rm dim}(\WW_h^+ \times \VV_h^-) = 39\,367\,938$. The dimension of the even part of the solution space is ${\rm dim}(\WW_h^+) = 11\,732\,226$. 
The tolerance of the PCG algorithm is set to $10^{-7}$ in all our tests.
The computations are performed with MATLAB 2016b on a MacBook pro with Intel i7-6700HQ (2.6 GHz) and 16 GB of memory.
In Table~\ref{tab:smooth_err_loc} we plot for $N\in\{9,11\}$ and $h\in \{0.02,0.01,0.005\}$ the error
\begin{align*}
	e_h^2=\|\widehat u-w^{\ell,a}_h\|_{L^2(\D)}^2+\|\sgrad (\widehat u-w^{\ell,a}_h)^+\|_{L^2(\D)}^2,
\end{align*}
where $\widehat u$ is our reference solution. As can be seen from the table, the error is determined by the approximation properties of the discretization and the consistency error due to the absorbing layer and the reflection boundary conditions. 
The slight increase in the error for $N=9$ and $h\geq 0.01$ for decreasing $e^{-a\ell}\leq 1/8$ could be explained by the fact that we approximate the curved boundary $\partial\R$ by triangles.
For $N=9$ and the spatial reference grid with $h=0.005$, we, however, observe a saturation effect, i.e., the approximation error dominates the consistency error for $e^{-a\ell}\leq  1/8$. Depending on the spatial and angular grid, we observe exponential decay of the error as long as the approximation error is negligible, which is indicated by italic numbers in Table~\ref{tab:smooth_err_loc}. Furthermore, for $h=0.005$ and $N=11$, there is no approximation error, and the total error decays exponentially as predicted by the theoretical results of part one. 
\begin{table}
\caption{\label{tab:smooth_err_loc} Example 1: Error $e_h$ for different $P_N$-approximations and damping parameters $a$. The italic numbers indicates the turnover point at which the approximation error dominates for increasing $a$.}
\centering
\begin{tabular}{c c c c c c c c}
\toprule
			&	\multicolumn{3}{c}{$N=9$}		&	&		\multicolumn{3}{c}{$N=11$} \\
				\cmidrule{2-4} 								\cmidrule{6-8} 
$e^{-a\ell}$& $h=0.02$ 	& $h=0.01$	& $h=0.005$	&	& $h=0.02$ & $h=0.01$	& $h=0.005$\\
 \midrule
$15/16$		& 0.085		& 0.072		& 0.068		&	& 0.082		& 0.068		& 0.0643\\
$7/8$		& 0.073		& 0.059		& 0.054		&	& 0.071		& 0.055		& 0.0500\\
$3/4$		& 0.061		& 0.042		& 0.035		&	& 0.059		& 0.038		& 0.0296\\
$2/3$		& 0.057		& 0.036		& 0.028		&	& 0.055		& 0.031		& 0.0204\\
$1/2$		& \it{0.054}		& 0.031		& 0.021 	&	& \it{0.052}		& 0.025		& 0.0090\\
$1/4$		& 0.054		& \it{0.029}		& 0.019		&	& 0.051		& \it{0.023}		& 0.0023\\
$1/8$		& 0.054		& 0.029 	& \it{0.019} 	&	& 0.051		& 0.023 	& 0.0012\\ 
$1/16$		& 0.055		& 0.030 	& 0.019 	&	& 0.052		& 0.023 	& 0.0005\\ 
\bottomrule
\end{tabular}
\end{table}
\begin{table}
\caption{\label{tab:smooth_it_loc} Example 1: Iteration numbers and runtime in seconds in brackets of the PCG algorithm. The last row contains the number of degrees of freedom of $\WW_h^+$ in million.}
\centering
\begin{tabular}{c c c c c c c c}
\toprule
			&	\multicolumn{3}{c}{$N=9$}		&	&		\multicolumn{3}{c}{$N=11$} \\
				\cmidrule{2-4} 								\cmidrule{6-8} 
$e^{-a\ell}$& $h=0.02$ 	& $h=0.01$	& $h=0.005$	&	& $h=0.02$ & $h=0.01$	& $h=0.005$\\
 \midrule
$15/16$		& 277 (104)	& 305 (518)	& 315 (2926)&	& 316 (184)	& 350 (993)	& 364 (5214)\\
$7/8$		& 195 (75)	& 215 (394)	& 222 (2079)&	& 220 (134)	& 246 (205)	& 256 (3611)\\
$3/4$		& 137 (51)	& 149 (259)	& 154 (1434)&	& 153 (89)	& 172 (478)	& 179 (2475)\\
$2/3$		& 117 (44)	& 127 (224)	& 131 (1203)&	& 132 (75)	& 146 (400)	& 152 (2087)\\
$1/2$		& 90 (35)	& 99 (178)	& 103 (947) &	& 102 (62)	& 115 (324)	& 119 (1675)\\
$1/4$		& 68 (27)	& 73 (124)	& 75 (716)	&	& 76 (47)	& 85 (234)	& 87 (1231)\\
$1/8$		& 56 (23)	& 64 (116) 	& 68 (647) 	&	& 62 (39)	& 73 (210) 	& 79 (1124)\\ 
$1/16$		& 57 (23)	& 70 (126) 	& 77 (725) 	&	& 63 (40)	& 78 (217) 	& 88 (1248)\\ 
\midrule
dofs		&0.51M		& 2.01M		&8.00M		& 	& 0.74M		& 2.94M		& 11.73M\\		
\bottomrule
\end{tabular}
\end{table}
In Table~\ref{tab:smooth_it_loc}, we show the iteration numbers of the PCG algorithm and the corresponding runtimes, which scale almost linearly in the number of degrees of freedom. For our example, we observe that the iteration numbers for different discretizations vary only slightly. To the best of our knowledge, a full analysis of such a preconditioner is, however, not available.
Furthermore, we observe that the preconditioned conjugate gradient algorithm converges faster for moderate to large absorption than for small absorption, which clearly is a favorable feature in the context of our perfectly matched layers approach. 
This convergence behavior can be explained by the exponential decay of the solution in the absorbing but non-scattering layer.
\subsection{Example 2: Non-smooth lattice problem}
The potential of the proposed method to solve large-scale problems is briefly illustrated by computational experiments for a lattice problem, which is used as a test case for simulating the core of a nuclear reactor in the nuclear engineering communities. The geometric setup of the problem, the absorption and the scattering parameters are depicted in Figure~\ref{fig:setup_check}. The source is defined as
$q(\r,\s)=1$ for $3\leq \r_1, \r_2\leq 4,\, \s\in\S$ and $q(\r,\s)=0$ else.
Since $\R=(0,7)\times(0,7)$ and $\R^\ell=(-1,8)\times(-1,8)$, we have $\ell=1$.
We again compute a reference solution for $e^{-a\ell}=1/32$ on a reference grid, which is chosen to consist of $332\, 929$ vertices, and $663\,552$ triangles, and spherical harmonics of order up to and including $N=31$, i.e., $1\,024$ Fourier coefficients. 
The number of degrees of freedom for approximating the solution on the reference grid is $515\,488\,240$, which amounts to nearly 4GByte of memory to just store the  even and odd parts of the solution. For the even part we have $165\,132\,784$ degrees of freedom. For this computation we used our MATLAB implementation on a workstation with an Intel Xeon CPU E5-2620 with 2.40GHz and 64 GB of memory. 
\begin{figure}
	\centering
	\includegraphics[width=0.49\textwidth]{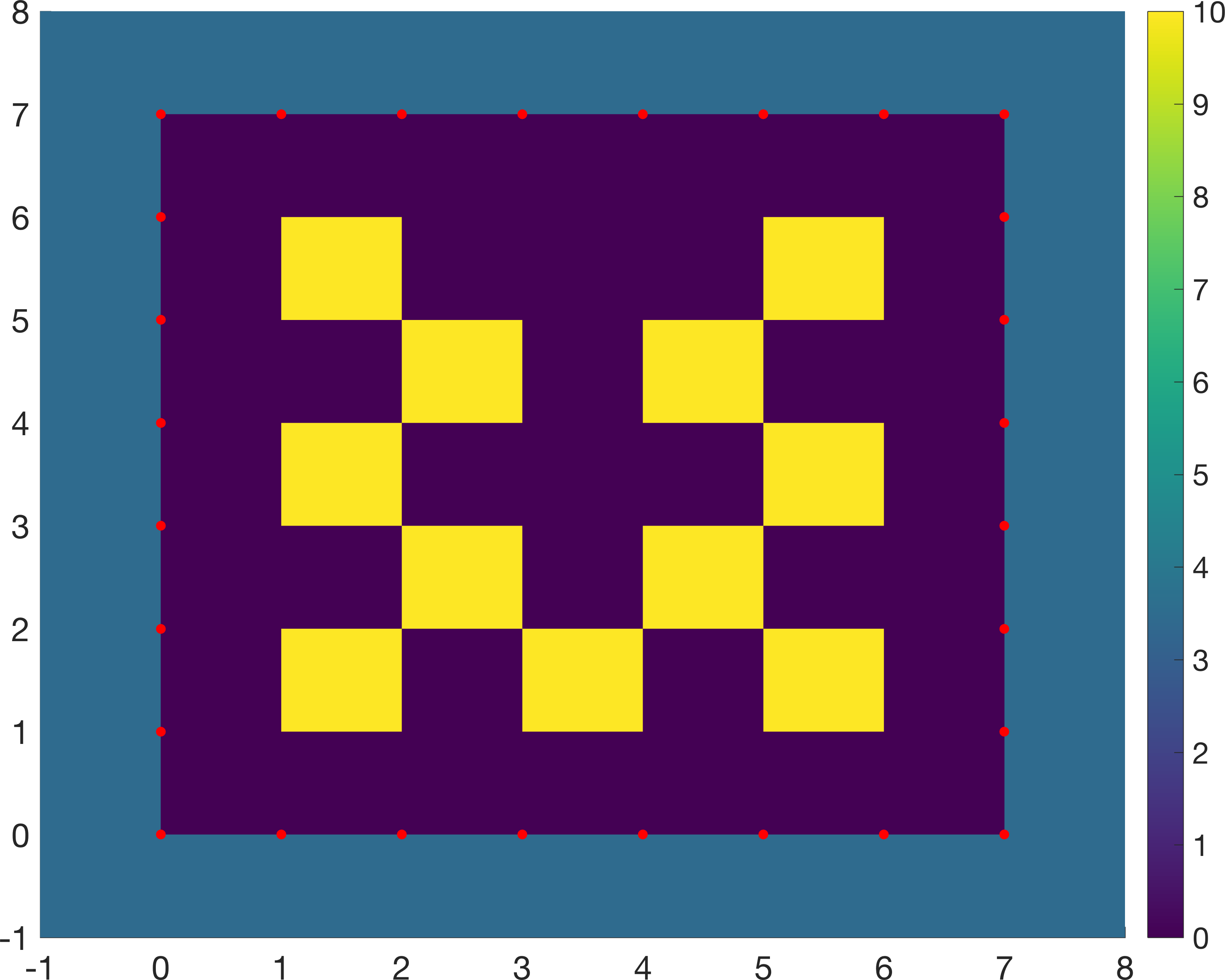}
	\includegraphics[width=0.49\textwidth]{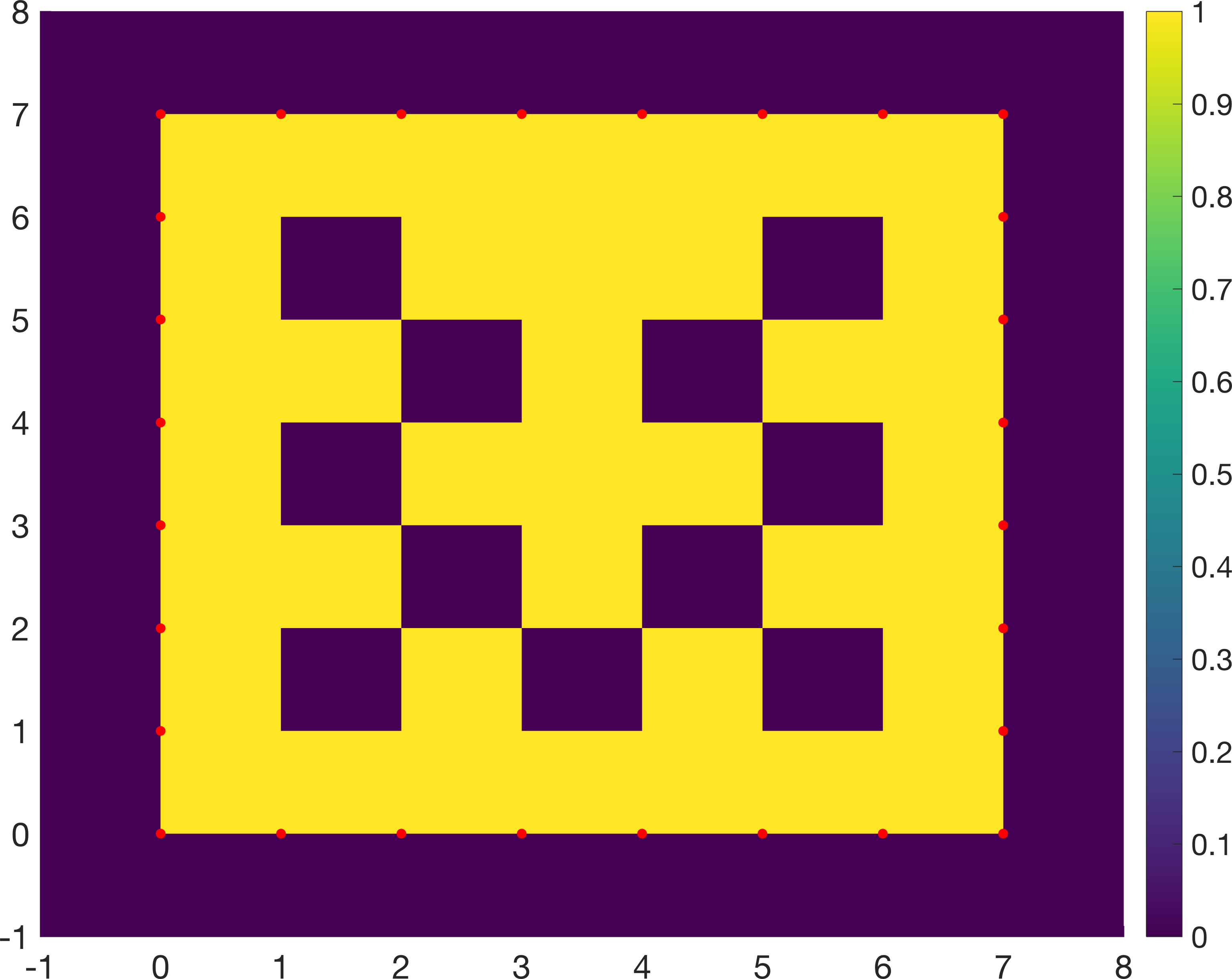}
	\caption{\label{fig:setup_check} Example 2: Sketch of the computational setup. Left: Extended absorption parameter with $e^{-a}=1/32$ on $\R^\ell\setminus\R$. Right: Extended scattering coefficient on $\R^\ell$. The domain $\R$ is enclosed by the dotted line.}
\end{figure}
Table~\ref{tab:check_err_loc_max_N} shows the error for a $P_N$-FEM approximation with $N=31$ for different damping parameters, which can be seen to decay exponentially as predicted by theory.
The PCG iteration numbers decreased from $270$ for $e^{-a\ell}=1/2$ to $226$ for $e^{-a\ell}=1/32$.
Moreover, despite the fact, that the absorption parameter has a quite strong jump discontinuity, the method does not produce oscillations, see Figure~\ref{fig:check_decay}. 
\begin{table}
	\centering
	\caption{\label{tab:check_err_loc_max_N} Example 2: Error $e_h$ and iteration numbers of the PCG algorithm on the reference grid for different absorption parameters $a$.}
\begin{tabular}{c c c c c c}
	\toprule
	$e^{-a\ell}$	&  $1/2$ & $1/4$ & $1/8$ & $1/16$	& $1/32$\\ 
	\midrule
$e_h\times 1000$	&   1.463 & 0.681	& 0.346	& 0.141	& 0.00\\ 
\bottomrule
\end{tabular}
\end{table}
\begin{figure}
	\centering
	\includegraphics[width=.75\textwidth]{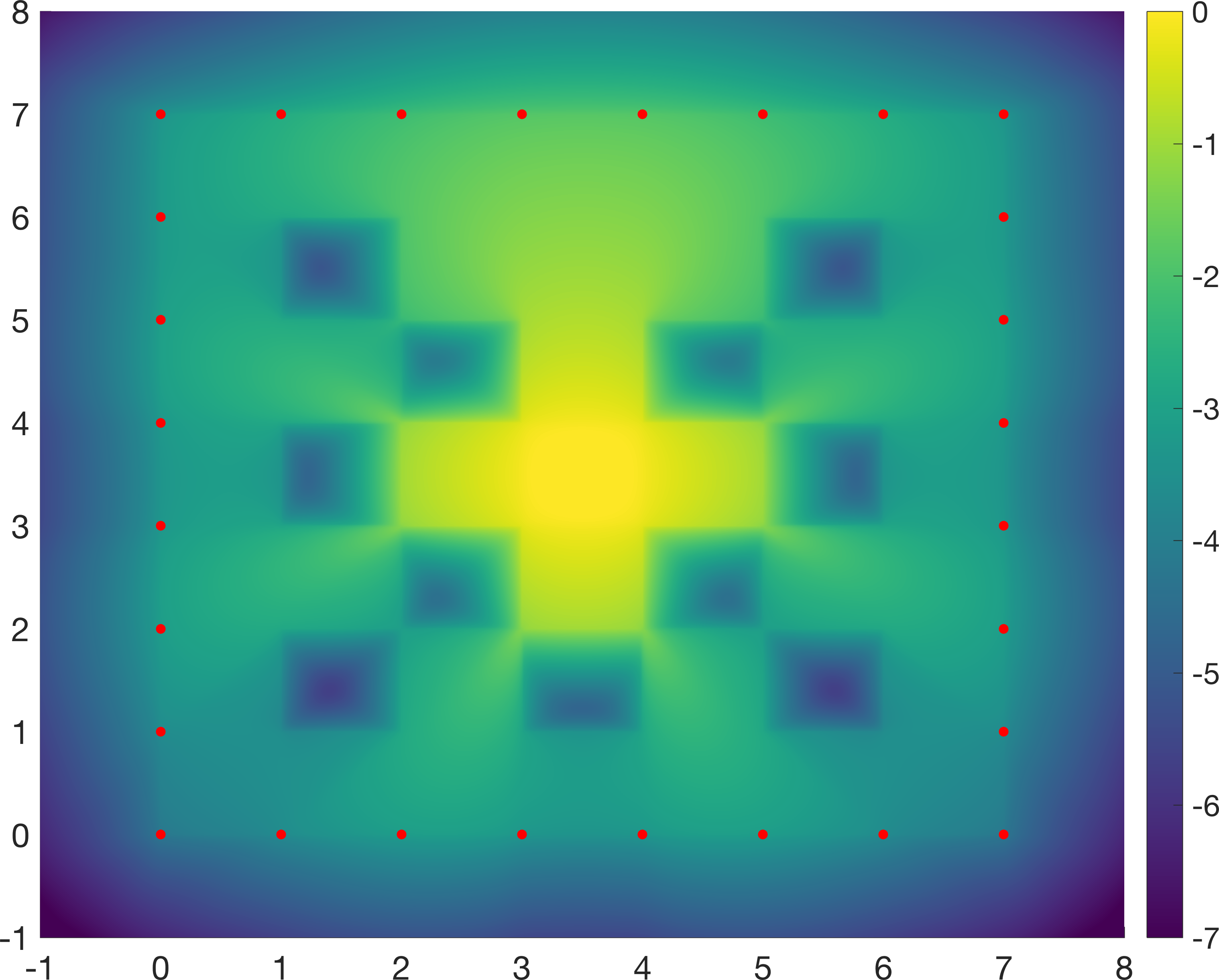}
	\caption{Example 2:  $\log_{10}$-plot of the angular average of the reference solution. 
	\label{fig:check_decay}}
\end{figure}
\section{Discussion and further applications}\label{sec:discussion}
We have presented a perfectly mat\-ched layer approach for the efficient treatment of vacuum boundary conditions in radiative transfer. 
The choice of reflection boundary condition on the boundary of the extended domain was specifically tailored to obtain a mixed variational formulation that leads to sparse linear systems, and which can be implemented easily.
In view of our detailed error estimates, it seems possible to analyze different artificial boundary conditions for the extended problem as well. As mentioned in the introduction, a relevant case are periodic boundary conditions, which in turn can be used to develop pseudospectral methods, see \cite{PowellCoxArridge18}.
Besides different boundary conditions it seems possible to generalize our results to non-constant extensions of the absorption coefficient as long as sufficient decay of the solution within the absorbing layer is guaranteed.
Finally, let us shortly comment on a further possible application of the theoretical setup.
\paragraph{Least-Squares formulations}
A powerful method for the solution of first order equations is the least-squares approach that has been developed for the radiative transfer equation in \cite{ManResSta00}, and has been widely used \cite{GrellaSchwab2011a,WidHiptSchw08}. The basic approach is to minimize the functional
\begin{align*}
	\| \sgrad u + \C u -q\|_{L^2(\D)}^2 + \|u\|_{L^2(\partial\D_-;|\sn|)}^2 \to \min!
\end{align*}
over the space $\W^2=\{v\in W^2(\D): v_{\mid\partial\D_-}\in L^2(\partial\D_-;|\sn|)\}$. Here, the homogeneous inflow boundary conditions are approximated by incorporating the boundary functional $\|u\|_{L^2(\partial\D_-;|\sn|)}$. As mentioned in the introduction, the numerical approximation of such half-space integrals makes the numerical realization of the minimization problem difficult. 
Based on the approach of this paper, it is natural to investigate the following modified least-squares problem
\begin{align*}
	\| \sgrad \bar w^{\ell,a} + \C^{\ell,a} \bar w^{\ell,a} -q^\ell\|_{L^2(\D^\ell)}^2 + \|\bar w^{\ell,a}\|_{L^2(\partial\D^\ell)}^2 \to \min!,
\end{align*}
where the minimum is sought in the space $\W^\ell=\{v\in W^2(\D^\ell): v_{\mid\partial\D^\ell}\in L^2(\partial\D^\ell)\}$. This is currently under investigation by the authors.
\section*{Acknowledgements}
HE and MS acknowledge financial support for a one week research visit of HE at the University of Twente by 4TU Centre of Competence "Fluid and Solid Mechanics".
\bibliographystyle{siamplain} 
\bibliography{absorbing_rte}

\begin{thebibliography}{10}

\bibitem{Agoshkov98}
{\sc V.~Agoshkov}, {\em Boundary Value Problems for Transport Equations},
  Modeling and Simulation in Science, Engineering and Technology, Birkh\"auser,
  Boston, 1998.

\bibitem{ArridgeEggerSchlottbom13}
{\sc S.~Arridge, H.~Egger, and M.~Schlottbom}, {\em Preconditioning of complex
  symmetric linear systems with applications in optical tomography}, Appl.
  Numer. Math., 74 (2013), pp.~35--48.

\bibitem{Arridge99}
{\sc S.~R. Arridge}, {\em Optical tomography in medical imaging}, Inverse
  Problems, 15 (1999), pp.~R41--R93,
  \url{https://doi.org/10.1088/0266-5611/15/2/022}.

\bibitem{Berenger94}
{\sc J.-P. Berenger}, {\em A perfectly matched layer for the absorption of
  electromagnetic waves}, J. Comput. Phys., 114 (1994), pp.~185--200,
  \url{https://doi.org/10.1006/jcph.1994.1159}.

\bibitem{CaseZweifel67}
{\sc K.~M. Case and P.~F. Zweifel}, {\em Linear transport theory},
  Addison-Wesley, Reading, 1967.

\bibitem{Chandrasekhar60}
{\sc S.~Chandrasekhar}, {\em Radiative Transfer}, Dover Publications, Inc.,
  1960.

\bibitem{ChangLee2003}
{\sc B.~Chang and B.~Lee}, {\em A multigrid algorithm for solving the
  multi-group, anisotropic scattering {B}oltzmann equation using first-order
  system least-squares methodology}, Electron. Trans. Numer. Anal., 15 (2003),
  pp.~132--151.

\bibitem{DautrayLions6}
{\sc R.~Dautray and J.~L. Lions}, {\em Mathematical Analysis and Numerical
  Methods for Science and Technology, Evolution Problems {II}}, vol.~6,
  Springer, Berlin, 1993.

\bibitem{DuderstadtMartin79}
{\sc J.~J. Duderstadt and W.~R. Martin}, {\em Transport Theory}, John Wiley \&
  Sons, Inc., New York, 1979.

\bibitem{EggerSchlottbom12}
{\sc H.~Egger and M.~Schlottbom}, {\em A mixed variational framework for the
  radiative transfer equation}, Math. Mod. Meth. Appl. Sci., 22 (2012),
  p.~1150014.

\bibitem{EggerSchlottbom2014Lp}
{\sc H.~Egger and M.~Schlottbom}, {\em An {$L^p$} theory for stationary
  radiative transfer}, Appl. Anal., 93 (2014), pp.~1283--1296,
  \url{https://doi.org/10.1080/00036811.2013.826798}.

\bibitem{GrellaSchwab2011a}
{\sc K.~Grella and C.~Schwab}, {\em Sparse tensor spherical harmonics
  approximation in radiative transfer}, J. Comput. Phys., 230 (2011),
  pp.~8452--8473, \url{https://doi.org/10.1016/j.jcp.2011.07.028}.

\bibitem{Hagstrom99}
{\sc T.~Hagstrom}, {\em Radiation boundary conditions for the numerical
  simulation of waves}, in Acta numerica, 1999, vol.~8 of Acta Numer.,
  Cambridge Univ. Press, Cambridge, 1999, pp.~47--106,
  \url{https://doi.org/10.1017/S0962492900002890}.

\bibitem{LewisMiller84}
{\sc E.~E. Lewis and W.~F. Miller~Jr.}, {\em Computational Methods of Neutron
  Transport}, John Wiley \& Sons, Inc., New York Chichester Brisbane Toronto
  Singapore, 1984.

\bibitem{ManResSta00}
{\sc T.~A. Manteuffel, K.~J. Ressel, and G.~Starke}, {\em A boundary functional
  for the least-squares finite-element solution for neutron transport
  problems}, SIAM J. Numer. Anal., 2 (2000), pp.~556--586.

\bibitem{MarchukLebedev86}
{\sc G.~I. Marchuk and V.~I. Lebedev}, {\em Numerical Methods in the Theory of
  Neutron Transport}, Harwood Academic Publishers, Chur, London, Paris, New
  York, 1986.

\bibitem{Modest}
{\sc M.~F. Modest}, {\em Radiative Heat Transfer}, Academic Press, Amsterdam,
  second~ed., 2003.

\bibitem{PowellCoxArridge18}
{\sc S.~Powell, B.~T. Cox, and S.~R. Arridge}, {\em A pseudospectral method for
  solution of the radiative transport equation}, tech. report, 2018.
\newblock arXive:1801.06364.

\bibitem{TervoKolmonen2002}
{\sc J.~Tervo and P.~Kolmonen}, {\em Inverse radiotherapy treatment planning
  model applying {B}oltzmann-transport equation}, Math. Models Methods Appl.
  Sci., 12 (2002), pp.~109--141,
  \url{https://doi.org/10.1142/S021820250200157X}.

\bibitem{Vladimirov61}
{\sc V.~S. Vladimirov}, {\em Mathematical problems in the one-velocity theory
  of particle transport}, tech. report, Atomic Energy of Canada Ltd. AECL-1661.
  translated from Transactions of the V.A. Steklov Mathematical Institute (61),
  1961.

\bibitem{WidHiptSchw08}
{\sc G.~Widmer, R.~Hiptmair, and C.~Schwab}, {\em Sparse adaptive finite
  elements for radiative transfer}, Journal of Computational Physics, 227
  (2008), pp.~6071--6105.

\bibitem{WrightArridgeSchweiger09}
{\sc S.~Wright, S.~R. Arridge, and M.~Schweiger}, {\em A finite element method
  for the even-parity radiative transfer equation using the {$P_N$}
  approximation}, in Numerical Methods in Multidimensional Radiative Transfer,
  G.~Kanschat, E.~Meink\"ohn, R.~Rannacher, and R.~Wehrse, eds.,
  Springer-Verlag, Berlin Heidelberg, 2009.

\end{thebibliography}
\end{document}